\documentclass[12pt,a4paper, oneside]{amsart}

\newif\ifdraft
\drafttrue 

\usepackage{fancyhdr}                       
\usepackage{amsfonts}
\usepackage{amssymb,amsmath,amsthm,eucal}   
\usepackage{graphicx,wrapfig}               
\usepackage{pdfpages}                       
\usepackage{textcomp} 			    

\usepackage{thmtools} 
\usepackage{thm-restate}

\usepackage{tikz-cd}
\usepackage{tikz}
\usetikzlibrary{matrix,arrows}
\usepackage{setspace}

\usepackage{hyperref}                       
\usepackage{enumerate}                      
\usepackage{array}                          
\usepackage{caption}                        

\usepackage{xcolor}
\usepackage{color}  
\usepackage{lipsum} 

\usepackage[final]{showlabels}             

\newtheorem{theorem}{Theorem}
\newtheorem{corollary}[theorem]{Corollary}
\newtheorem{lemma}[theorem]{Lemma}
\newtheorem{conjecture}[theorem]{Conjecture}
\newtheorem{proposition}[theorem]{Proposition}
\theoremstyle{definition}
\newtheorem{remark}[theorem]{Remark}
\newtheorem{assumption}[theorem]{Assumption}
\newtheorem{example}[theorem]{Example}
\newtheorem{question}[theorem]{Question}

\newtheorem{definition}[theorem]{Definition}

\def \C{{ \mathbb C }}

\def \Z{{ \mathbb Z }}
\def \P{{ \mathbb P }} 

\def \Q{{ \mathbb Q }}

\def \A{{ \mathbb A }}

\def \D{{ \mathcal D}}
\def \S{{ \mathfrak S}}
\def \CO{{ \mathcal{O} }}
\def \CA{{ \mathcal{A} }}
\def \CB{{ \mathcal{B} }}
\def \CC{{ \mathcal{C} }}

\def \h{{ \mathfrak{h} }}           

\def \Sym{{ 	\mathsf{Sym} }}

\def \Hilb{{ 	\mathsf{Hilb}}}
\def \Blow{{ 	\mathsf{Blow}}}
\def \Pic{{     \mathsf{Pic} }}
\def \Spec{{ 	\mathsf{Spec}}}

\def \Rep{{     \mathsf{Rep} }}
\def \ch{{    \mathsf{char}~ }}
\def \Gal{{     \mathsf{Gal} }}
\def \Burn{{    \mathsf{Burn}}}
\def \CH {{     \mathsf{CH}  }}
\def \K{{       \mathsf{K}  }}
\def \Chow{{    \mathsf{Chow_{\Q}}}}
\def \ChowZero{{    \mathsf{Chow_{\Q}^{0}}}}
\def \ChowArt{{    \mathsf{Chow_{\Q}^{Art}}}}
\def \Kthree{ S_{\mathrm{K3}} } 
\def \Mot{ { \mathsf{Mot} }}
\def \Blow{ { \mathsf{Blow} }}
\newcommand{\xdasharrow}[2][->]{
\tikz[baseline=-\the\dimexpr\fontdimen22\textfont2\relax]{
\node[anchor=south,font=\scriptsize, inner ysep=1.5pt,outer xsep=2.2pt](x){#2};
\draw[shorten <=3.4pt,shorten >=3.4pt,dashed,#1](x.south west)--(x.south east);
}
}

\def \W{ 	\mathbb{W}}	
\def \HY{ 	[S]} 		
\def \L{ 	\mathbb L}	
\def \One{ 	1 } 	 	
\def \X_#1 { 	\chi_{#1} } 
\def \KGro{ 	\mathsf{K}_0[\mathsf{Var}_{k}] } 	
\def \KCat{ 	\mathsf{K}_0[\mathsf{dg}\text{-}\mathsf{cat}_{k}] } 	
\def \Prop{     \mathsf{(P)}  }

\def \muCat{    \mu_{cat}   }
\def \muSB{     \mu_{sb}    }
\def \muGS{     \mu_{mot}   }

\def \YZYequation{
                \L^4[Z(Y)] &= [Y^{(4)}] - (1 + \L^2 + \L^4)[Y^{(3)}] + \L^2[Y][Y^{(2)}] - (\L^2 + \L^6)[Y^2] \\
                           &+ \L^4[Y^{(2)}] - (\L^{2} + \L^6 + \L^{10})[Y] \nonumber}

\numberwithin{equation}{section}  
\numberwithin{theorem}{section}
\numberwithin{remark}{section}
\numberwithin{example}{section}
\numberwithin{proposition}{section}
\numberwithin{lemma}{section}
\numberwithin{definition}{section}
\numberwithin{assumption}{section}
\numberwithin{question}{section}
\numberwithin{corollary}{section}
\numberwithin{conjecture}{section}
\allowdisplaybreaks

\title{Twisted cubics and quadruples of points on cubic surfaces}
\author{Pavel Popov}
\date{\today}

\begin{document}

\begin{abstract}
        We study relations in the Grothendieck ring of varieties which connect the Hilbert scheme of points on a cubic hypersurface $Y$ with a certain moduli space of twisted cubic curves on $Y$.
        These relations are generalizations of the ``beautiful'' $Y$-$F(Y)$ relation by Galkin and Shinder which connects $Y$ with the Hilbert scheme of two points on $Y$ and the Fano variety~$F(Y)$ of lines on $Y$.
        We concentrate mostly on the case of cubic surfaces. 
        The symmetries of $27$ lines on a smooth cubic surface give a lot of restrictions on possible forms of the relations.
\end{abstract}
\maketitle

\setcounter{secnumdepth}{3}
\tableofcontents

\section{Introduction} \label{Introduction}
\addtocontents{toc}{\protect\setcounter{tocdepth}{1}}


There are different kinds of geometric objects associated to a given algebraic variety.
Some of these objects come in families.
A certain moduli spaces which parameterise such objects again have an algebro-geometric structure.
They help to shed light on the hidden properties of the algebraic variety.
Different moduli spaces associated with one algebraic variety are closely related.
In this work, we will consider cubic hypersurfaces.
The most simple geometric objects on them are configurations of points and rational curves.
We will explore some connections between cubic hypersurfaces and moduli spaces of such objects on them.

\subsection{Lines on cubic hypersurfaces}
Consider a cubic hypersurface $Y\subset \P^{d+1}$ over a field $k$. 
Let $F(Y)$ denote the reduced Fano scheme of lines on $Y$.
For smooth $Y$ the Fano scheme $F(Y)$ is a smooth projective variety of dimension $2d-4$, \cite{AK77}. 
\begin{question}
How does the geometry of $Y$ and $F(Y)$ relate?
\end{question}

The following answer is given in \cite{GS14}.
Passing a line through two generic points on $Y$ one can obtain the third point on $Y$. 
In other words, there is a rational map from the Hilbert scheme $Y^{[2]}$ of two points to $Y$.
This map is a composition of two.
The first one is a birational isomorphism $\phi$ between $Y^{[2]}$ and the variety of pairs 
$$W=\{(l,p) | \text{ a line } l\subset \P^{d+1}, \text{ a point } p\in l\cap Y\}.$$ 
It sends a pair of points to the line passing through them and the third point of intersection with $Y$. 
The other is a Zariski locally trivial fibration from $W$ to $Y$ with fiber $\P^d$.
It sends a pair $(l, p)\in W$ to the point $p$.
\begin{equation}\label{eq:phi}
\begin{gathered}
  \begin{tikzpicture}
    \matrix (m) [matrix of math nodes, row sep=1em,
column sep=1em]
        {    &  W  & \\
        Y^{[2]} && Y \\ };
	\path[<->, dashed]    (m-2-1) edge node[auto]{$\phi$} (m-1-2);
  \path[->]             (m-1-2) edge (m-2-3);
  \path[->, dashed]     (m-2-1) edge (m-2-3);
\end{tikzpicture}
\end{gathered}
\end{equation}

The \emph{Grothendieck ring} $\KGro$ of $k$-varieties as an abelian group is generated by the isomorphism classes of reduced, quasi-projective $k$-schemes. The relations are
$$[X] = [Y] + [X\setminus Y],$$
whenever $Y$ is a closed subscheme of $X$. See Section \ref{KGro} for details.

A careful analysis of the locus of indeterminacy of birational isomorphism $\phi$ gives the following ``beautiful'' $Y$-$F(Y)$ relation \cite{GS14} in the Grothendieck ring of varieties $\KGro$:
\begin{equation}\label{The Y-F(Y) relation}
  [Y^{[2]}] = [\P^d][Y]+\L^2[F(Y)],
\end{equation}
where $\L:=[\A^1]\in \KGro$ is the Lefschetz class.
This birational isomorphism was also studied in \cite{Voi15} and used to show that $\CH_0(Y)$ is universally trivial for certain smooth cubic hypersurfaces of dimensions $3$ and $4$.
 
The relation \eqref{The Y-F(Y) relation} is beautiful since, first of all, it has the same elegant form for any cubic hypersurface over any field $k$. Secondly, it is a relation in any \emph{realization} of the Grothendieck ring of varieties $\KGro$. While the ring $\KGro$ is too complicated one may consider a realization homomorphism (i.e. ring homomorphism) to another ring $R$. 
There are realizations in the classes of stable birational equivalence, in the Grothendieck ring of categories, in the Grothendieck ring of rational Chow motives etc. In some literature realization homomorphisms called \emph{motivic measures}. See Section~\ref{KGro} for details.

Different kinds of invariants of the Fano variety $F(Y)$ may be computed using the $Y$-$F(Y)$ relation (\ref{The Y-F(Y) relation}). 
For example, one can easily compute the number of lines on real and complex smooth or singular cubic surfaces. 
On the other hand, one can compute the Hodge structure of $H^*(F(Y),\Q)$ for a smooth complex cubic $Y$ of an arbitrary dimension $d$. 
In \cite{DLR17} this relation is used to calculate zeta functions of the Fano variety $F(Y)$ for smooth cubic threefolds and fourfolds and answer some questions about lines on cubic hypersurfaces over finite fields.  

\begin{remark}
        Let $k$ be a field of characteristic $0$. 
	Denote by $\Chow$ the Grothendieck's category of Chow $k$-motives, \cite{Sch94}.
        The \emph{Gillet--Soul\'e motivic realization} homomorphism
        $$\muGS\colon\KGro \rightarrow \K_0(\Chow)$$
	sends the class $[X_k]$ of a smooth projective variety to the class $[\h(X_k)]$ of its motive.
	The~$Y\text{-}F(Y)$ relation (\ref{The Y-F(Y) relation}) implies a relation in $\K_0(\Chow)$.
        In \cite{La16} it was shown that the $Y$-$F(Y)$ relation lifts to an equivalence in the category of Chow motives:
	$$\h(Y^{[2]}) \cong \h(F(Y))(2)\oplus \overset{d}{\underset{i=0}{\bigoplus}} \h(Y)(i).$$
\end{remark}

The $Y$-$F(Y)$ relation also implies a relation on the level of categories. 
More precisely, the \emph{categorical realization} homomorphism
$$\muCat\colon \KGro \rightarrow \KCat$$
sends the class $[X_k]$ of a smooth projective variety to the class $[\D(X_k)]$ of the bounded derived category of coherent sheaves on $X_k$ in the Grothendieck ring of pre-triangulated dg-categories,~\cite{BLL04}.
The relations in this ring come from semi-orthogonal decompositions.
We recall the details in Section~\ref{KGro}.

In the case of smooth cubic fourfolds the class $[\D(Y)]$ is equal to $[\CA_Y] + 3$, where $\CA_Y$ is a certain subcategory in $\D(Y)$.
By many features the category $\CA_Y$ looks like the derived category of a $K3$-surface, \cite{Kuz08}. The $Y$-$F(Y)$ relation allows to express $[\D(F(Y))]$ through the class~$[\CA_Y]$:
$$[\D(F(Y))] = [\Sym^2_{GK} \CA_Y],$$
where $\Sym^n_{GK}\CC$ are symmetric power operations in the category $\KCat$, \cite{GK14}.
For certain special cubic fourfolds this relation lifts to an equivalence of categories. 
This follows from the isomorphism $F(Y) \cong \Hilb^2(S)$, \cite{BD85}, where $S$ is a $K3$-surface.
Sergey Galkin conjectured an equivalence $\D(F(Y))\cong\Sym^2_{GK}\CA_Y$ for all smooth cubic fourfolds.
We discuss this in detail in Section~\ref{Speculations for cubic fourfolds}.


\subsection{Configurations of points on cubic hypersurfaces}
The $Y$-$F(Y)$ relation (\ref{The Y-F(Y) relation}) uses the Hilbert scheme of two points on $Y$ to give a connection between $Y$ and the Fano variety~$F(Y)$ of lines on $Y$.
It also can be rewritten in the form (\ref{The Y-F(Y) relation Sym}) (see Example \ref{ex:yfy} below) using the symmetric square $Y^{(2)}$ of $Y$ instead of the Hilbert scheme.
In this paper we consider both $Y^{[n]}$ and $Y^{(n)}$ as ``moduli spaces of configurations of points on $Y$''.
Some times we will also consider the variety $Y^{\{n\}}\subset Y^{(n)}$ of unordered configurations of $n$ different points.

The symmetric powers $Y^{(n)}$ are related to $K$-points on $Y$ for finite field extensions $K/k$. 
The basic question on $K$-points is about their existence on $Y$.
The existence of a $k$-point is a stable birational invariant.
According to \cite{LL03} the quotient ring $\KGro/(\L)$ as an abelian group is freely generated by classes of stable birational equivalence and any relation $[X]\equiv [Y]~(mod~\L)$ implies a stable birational equivalence of smooth projective connected varieties $X$ and $Y$. 
In particular the $Y$-$F(Y)$ relation (\ref{The Y-F(Y) relation}) implies that $Y$ and $Y^{[2]}$ are stably birationally equivalent. 
This also follows from the diagram \eqref{eq:phi}.
If on a smooth cubic $Y$ there is a point defined over a quadratic field extension then there is a point defined over the ground field.
We discuss this in details in Section \ref{SB obstruction}. See also Section \ref{KGro} for the full form of the result from \cite{LL03}.

In May of 2015 J\'anos Koll\'ar asked Sergey Galkin the following question:
\begin{question}\label{Q4}(\emph{J\'anos Koll\'ar})
Are there other "beautiful" formulae which express higher Hilbert schemes of points on a cubic hypersurface $Y$ (or symmetric powers of $Y$) through lower Hilbert schemes of points and something else? Is there some formula which connects the fourth Hilbert scheme $Y^{[4]}$ and $Y$ in the Grothendieck ring of varieties?
\end{question}

The positive answer to this question may help to solve the following conjecture by Cassels and Swinnerton-Dyer stated in the 1970s, \cite{Cor76a}:

\begin{conjecture}\label{CSDconj}
    Let $Y$ be a cubic hypersurface over a field $k$. 
    Suppose $Y$ has a point defined over an extension field $K/k$ of degree $n$ prime to $3$. 
    Then $Y$ also has a point over the ground field $k$.
\end{conjecture}

The case of $n=2$ follows from the $Y$-$F(Y)$ relation (\ref{The Y-F(Y) relation}) and the case of $n=4$ is an open question. 
Daniel Coray, \cite{Cor76a}, \cite{Cor76b}, shows that this conjecture is true in the case of cubic surfaces over local fields and in the case of singular cubic surfaces. 
Also he shows that if a cubic surface has a point over a field extension $K/k$ of degree prime to $3$ then it has a point over a field of degree $1,4$ or $10$.
Thereby Coray already reduced the general conjecture to the cases of $n=4$ and $n=10$. 
In this work we will concentrate on the case of $n=4$.

In order to study some questions about all the symmetric powers of $X$ together one can organise them in a power series, so called \emph{Kapranov's zeta function}, in some literature it is called ``motivic zeta function'',
$$ Z_{Kap}(X, t) = \sum_{n=0}^{\infty} X^{(n)}t^n \in \KGro[[t]]. $$

If $k$ is a finite field, $\mu\colon [X]\mapsto |X(k)|$ defines a realization $\mu\colon\KGro\rightarrow \Z$, where $|X(k)|$ denotes the number of $k$-points on $X$. 
The image under $\mu$ of the Kapranov's zeta function is the Hasse-Weil zeta function which is rational as a function of $t$ by Dwork's theorem.
In~\cite{Kap00} Kapranov asked whether the rationality holds for $Z_{Kap}(X,t)$.
He proved rationality in the case of curves.
In \cite{LL03} the authors prove that in general Kapranov's zeta function is not rational for surfaces.
See \cite{LL18} for a summary and recent results on the rationality of Kapranov's zeta function.
The case of a non-rational cubic surface is an open question.

\subsection{Beautiful formulae}\label{sec:beau}
Motivated by $Y$-$F(Y)$ relation (\ref{The Y-F(Y) relation}) we are going to search for connections between $Y$ and higher symmetric powers $Y^{(n)}$ of $Y$ in the Grothendieck ring of varieties $\KGro$.
Let us introduce some definitions before stating the precise form of Question~\ref{Q4} that we are going to answer. 

Consider a polynomial expression $P$ with formal symbols $[Y]=[Y^{(1)}],~[Y^{(n)}],~[X_i(Y)],~\L$.
We will call such an expression \emph{a formula with varieties $X_i(Y)$} if it is linear in variables $[X_i(Y)]$.
We assign degree to a polynomial expression $P$ by assigning degree $1$ to $[Y]$, $n$ to $[Y^{(n)}]$ and degree $0$ to other variables.
We say that $P$ is \emph{a formula of degree $n$} if $P$ has degree $n$, i.e. it consists of monomials of degree $n$ and less. 

Let $\Prop$ be some property: \emph{any}, \emph{smooth}, \emph{singular}, \emph{singular of some type}, \emph{of dimension $n$} etc. 
Let $X_1(Y),\dots, X_m(Y)$ be some varieties constructed naturally for any cubic hypersurface $Y$ with property $\Prop$ over any field $k$. 
We can evaluate \emph{a formula $P$ with varieties~$X_1(Y),\dots, X_m(Y)$} in the Grothendieck ring of varieties $\KGro$ by sending $[Y]$ to the class of $Y$, $[Y^{(n)}]$ to the classes of symmetric powers $\Sym^n(Y)$, $\L$ to the class of an affine line $\A^1$ and $[X_i(Y)]$ to the class of the variety $X_i(Y)$.
Moreover for any realization $\mu\colon\KGro\rightarrow R$, we can evaluate a formula $P$ in the ring $R$.
We will say that $P$ is \emph{a formula in the ring $R$ for cubic hypersurfaces $Y$ of type $\Prop$ with varieties $X_1(Y),\dots, X_m(Y)$} if the $\mu$ realization of $P$ is zero for any cubic surface $Y$ with property $\Prop$ over any field $k$ (maybe except some fields of small characteristic).
We will call a formula $P$ \emph{homogeneous} if there are no additional varieties~$X_i(Y)$.

\begin{definition}\label{def:beautiful}
  We will say that a formula $P$ for cubic hypersurfaces $Y$ of type $\Prop$ is \emph{beautiful} if it is a formula in the Grothendieck ring of varieties $\KGro$.
	In other words $P$ is a formula in the ring $R$ for all realizations $\mu:\KGro\rightarrow R$.
	This means that $P$ is zero after evaluation in~$\KGro$ for all cubic hypersurfaces $Y$ with property $\Prop$ over any field $k$ (maybe except some fields of small characteristic).

\end{definition}

\begin{example}\label{ex:yfy}
        For a smooth cubic hypersurface the $Y$-$F(Y)$ relation (\ref{The Y-F(Y) relation}) may be rewritten in the following form, \cite{GS14}:
        \begin{equation}\label{The Y-F(Y) relation Sym}
	  [Y^{(2)}] = (1+\L^d)[Y] + \L^2[F(Y)].
        \end{equation}
        This is an example of a \emph{beautiful formula of degree $2$ for smooth cubic hypersurfaces $Y$ with the Fano variety $F(Y)$}.
\end{example}

\begin{remark}
        Sometimes we will specify that the above-defined formulae are \emph{in $\Sym$-form}. 
        We make a parallel definition for formulae \emph{in $\Hilb$-form} by using Hilbert schemes $Y^{[n]}$ instead of $Y^{(n)}$.
        In the case of smooth surfaces $S$ any formula in $\Sym$-form imply some formula in~$\Hilb$-form and vice versa. 
        More precisely, for $\mathsf{char}~k>n$ subrings in $\KGro$ generated by $\L,[S^{(1)}],\dots,[S^{(n)}]$ and by $\L,[S^{[1]}],\dots,[S^{[n]}]$ coincide thanks to the G\"ottsche formula (\ref{Gottsche-formula}). 
\end{remark}

\begin{remark}
        A beautiful formula is still a beautiful formula after multiplication by some class~$[X]$. 
        However, it may become less powerful.
        For example, after multiplication by $\L$, we can not derive any information about stable birational geometry.
\end{remark}

\begin{question}\label{QPrecise}
        Is there a \emph{beautiful formula of degree $4$ for cubic hypersurfaces $Y$ with some varieties $X_i(Y)$} ? 
\end{question}

The Fano variety $F(Y)$ is not enough to obtain any relations due to the following result which we will show in Section \ref{The S-Z(S) relation}.
\begin{restatable}{theorem}{Fthreefourthm}\label{thm:F34}
There are no beautiful formulae of degree $3$ or $4$ for smooth cubic surfaces $S$ with the Fano variety $F(S)$.
\end{restatable}

\subsection{Twisted cubics on cubic hypersurfaces}
Motivated by the Fano variety of lines one may consider moduli spaces of smooth rational curves of arbitrary degree $d$ on $Y$. 
For $d\geqslant 2$ these spaces are no longer compact. 
To be specific we let $M_d(Y)$ denote the compactification in the Hilbert scheme $\Hilb^{dn+1}(Y)$, where $dn+1$ is the Hilbert polynomial in variable $n$. 

The space $M_2(Y)$ does not provide anything new in comparison with $M_1(Y)=F(Y)$. 
Indeed, any rational curve $C$ of degree $2$ on $Y$ spans a two-dimensional linear subspace $\P^2\subset \P^{d+1}$ which, in turn, cuts out a plane curve of degree $3$ from $Y$.
As this curve contains $C$, it must have a line~$L$ as a residual component.
Mapping $[C]$ to $[L]$ defines a natural morphism $M_2(Y) \rightarrow M_1(Y)$.
The fiber is the space of planes in $\P^{d+1}$ which contain the line.

The geometry of $M_3(Y)$ is much more interesting.
Let $Y\subset \P^5$ be a smooth cubic hypersurface that does not contain a plane. 
In \cite{LLSvS} the authors show that the moduli space~$M_3(Y)$ of generalized twisted cubic curves on $Y$ is a smooth and irreducible projective variety of dimension~$10$.
Any twisted cubic $C$ on $Y$ spans a three-dimensional linear subspace~$\P^3$ which in turn cuts out a cubic surface $S$. 
Twisted cubics on a smooth cubic surface come in $72$ two-dimensional families.
These families become $\P^2$-families after the compactification by generalized twisted cubics.
In case of cubic surfaces with at most rational double point singularities generalized twisted cubics also form $\P^2$-families, \cite{LLSvS}.  
The intersection $S=\P^3\cap Y$ can have more complicated singularities and generalized twisted cubics on such $S$ no longer form two-dimensional families.
Nevertheless, it is possible to make a contraction of the whole moduli space $M_3(Y)$ with the fiber $\P^2$.

In \cite{LLSvS} the authors construct a two-step contraction $M_3(Y)\overset{a}{\rightarrow}Z'(Y)\overset{\sigma}{\rightarrow} Z(Y)$. 
Where~$a$ is a $\P^2$-fiber bundle and $\sigma$ is the blow-up of $Z(Y)$ along $Y$. 
When $k=\C$ the variety $Z(Y)$ is a smooth eight-dimensional holomorphically symplectic manifold. 
We will refer to it as the \emph{LLSvS variety} $Z(Y)$ and morally understand it as a ``moduli space of $\P^2$-families of generalized twisted cubics on $Y$''.
In the case of a cubic surface $S$ with at most rational double singularities we will denote by $Z(S)$ the variety of $\P^2$-families of generalized twisted cubics on $S$ and also refer to it as the \emph{LLSvS variety} $Z(S)$. 
We recall results on $Z(Y)$ and~$Z(S)$ in Section~\ref{LLSvS}.

In the Grothendieck ring of varieties $[M_3(Y)]$ can be expressed through $[Y]$, $[Z(Y)]$ and $\L$. 
$$[M_3(Y)] = [\P^2]([Z(Y)] - [Y] + \P^3[Y])$$
That is why any beautiful formula with $M_3(Y)$ implies a beautiful formula with the LLSvS variety $Z(Y)$. Further we will use the LLSvS variety $Z(Y)$ when we talk about moduli spaces of twisted cubics on $Y$. It the case of cubic surfaces $S$ with at most rational double point singularities by \cite[Theorem~2.1]{LLSvS} (see Theorem~\ref{gtc on sing surf} below) the moduli space of generalized twisted cubic curves on $S$ is~$Z(S)\times\P^2$.  

In the case of smooth cubic fourfolds, there is a similar connection between the Hilbert scheme of four points and $Z(Y)$ as between the second Hilbert scheme and $F(Y)$.
More precisely for certain smooth cubic fourfolds $Y$ without a plane the LLSvS variety $Z(Y)$ is deformation equivalent to $\Hilb^4(S)$ for some $K3$ surface.
Sergey Galkin conjectured that $\D(Z(Y))\cong\Sym^4_{GK} \CA_Y$. We discuss this in Section \ref{Speculations for cubic fourfolds}.
\begin{question}(\emph{Main question})
        Is there a \emph{beautiful formula of degree $4$ for cubic hypersurfaces $Y$ with the LLSvS variety $Z(Y)$}?
\end{question}

\subsection{Results}
In order to investigate the existence of beautiful formulae, we will use three different realizations from the Grothendieck ring of varieties $\KGro$ to other rings. 
\begin{itemize}
    
  \item The \emph{Gillet-Soul\'e motivic realization} $\muGS$ in the Grothendieck ring $\K_0(\Chow)$ of the category of Chow $k$-motives. 
  \item The \emph{Larsen-Lunts stable birational realization} $\muSB$ in the ring $\Z[SB_k]$ generated by classes of stable birational equivalence of smooth connected projective varieties.
  \item The \emph{Bondal-Larsen-Lunts categorical realization} $\muCat$ in the ring $\KCat$ generated by quasi-equivalence classes of pre-triangulated dg-categories modulo some relations coming from semi-orthogonal decompositions.
\end{itemize}
See Section \ref{KGro} for details.

In Section \ref{The S-Z(S) relation} we restrict ourselves to the case of cubic surfaces and use the Gillet-Soul\'e motivic realization.
Any twisted cubic or a generic quadruple of points on a cubic hypersurface~$Y$ spans~$\P^3\subset \P^{d+1}$ which in turn cuts out a cubic surface $S$. 
So one may hope to generalize a beautiful formula for $S$ of degree $4$ with the LLSvS variety $Z(S)$ to higher dimensions.
The case of five points on a surface is already too degenerate.    


The geometric N\'eron-Severi group $N(S):=\Pic(S\otimes \bar{k})$ of a smooth cubic surface $S$ has a big group of isometries, see Section~\ref{The 27 lines} for details. 
This group is the Weyl group $\W$ of type $E_6$ and its action on $N(S)$ can be obtained through the Galois action. 
These symmetries give a lot of restrictions on the possible beautiful formulae for $S$.  
We obtain in Section \ref{sec:smooth} the following results:
\begin{restatable*}{theorem}{mainthm}
\label{thm:main}
        The only possible form (up to multiplication) of a beautiful formula of degree~ $4$ for smooth cubic surfaces $S$ with the LLSvS variety $Z(S)$ is the following $S$-$Z(S)$ relation:
\begin{align}\label{S-Z(S) formula}
         \L^4 [Z(S)] &= [S^{(4)}] - (1 - \L +  \L^2)[S^{(3)}] - \L[S][S^{(2)}] + (\L+ \L^2+ \L^3)[S^2] \\
	 &- 2\L^2[S^{(2)}] - ( \L- \L^2 + \L^3- \L^4 + \L^5)[S] + (\L^2+ \L^4+ \L^6). \nonumber
\end{align}
More precisely it is a unique (up to multiplication) formula of degree $4$ that holds in $\K_0(\Chow)$ for smooth cubic surfaces $S$ with the LLSvS variety $Z(S)$.

In terms of Hilbert schemes this relation has the following form:
\begin{align}\label{S-Z(S) formula Hilb}
        \L^4[Z(S)] &= [S^{[4]}] - (1 - \L + \L^2)[S^{[3]}] - 2\L[S][S^{[2]}] + (2\L + \L^2 + 2\L^3)[S^2]\\
	&- 3\L^2[S^{[2]}] - (\L - 2\L^2 - 2\L^4 + \L^5 ) [S] + (\L^2 + \L^4 + \L^6 ). \nonumber
\end{align}
\end{restatable*}

\begin{restatable*}{theorem}{motivicmainthm}
        \label{thm:motivicmain}
        The $S$-$Z(S)$ relation \eqref{S-Z(S) formula Hilb} lifts to an equivalence in the category $\Chow$ of Chow $k$-motives:
\begin{align*}
  &\h(Z(S))(4)\oplus\h(S^{[3]}) \oplus\h(S^{[3]})(2) \oplus\h(S\times S^{[2]})(1)^{\oplus 2}\oplus \h(S^{[2]})(2)^{\oplus 3} \oplus
  \h(S)\oplus \h(S)(5) 
  \cong \\ &\h(S^{[4]}) \oplus \h(S^{[3]})(2)
  \oplus \h(S^2)(1)^{\oplus 2} \oplus \h(S^2)(2) \oplus \h(S^2)(3)^{\oplus 2}\oplus 
  \h(S)(2)^{\oplus 2}\oplus\h(S)(4)^{\oplus 2} 
\oplus \\& \Q(2) \oplus \Q(4) \oplus \Q(6).        
\end{align*}
\end{restatable*}

In Section~\ref{Singular cubic surfaces} we will use the same realization in the case of singular cubic surfaces with, at most, rational double point singularities.
We obtain the following analogs of the $S$-$Z(S)$ relation~(\ref{S-Z(S) formula}) in the case of singular cubic surfaces with one rational double point singularity of type~$A_1$ or~$A_2$.

\begin{restatable*}{lemma}{Aonegrolemma}\label{lem:SA1Gro}
  Let $S$ be a cubic surface over a field $k$ with one rational double point singularity of type $A_1$. Denote by $A(S)$ the zero-dimensional scheme (which may not have a $k$-point) of lines passing through the singular point, then: 
\begin{equation}\label{eq:SA1Gro}
  \KGro \ni [S] = 1 + [A(S)]\L + \L^2 + \Delta,
\end{equation}
where $\Delta = \P^1 - [C]$ and $C$ is a smooth conic (which also may not have a $k$-point).
\end{restatable*}
\begin{restatable*}{theorem}{Aonethm}\label{thm:A1} 
  Let $S$ be any singular cubic surface with one singularity of type $A_1$ such that~$\Delta = 0$, see Lemma~\ref{lem:SA1Gro}. The following relation is a beautiful formula $($i.e. it holds in~$\KGro)$ of degree $4$ for $S$ with the LLSvS variety $Z(S)$:
\begin{equation}\label{S-ZS-A1}
  \begin{split}
    \L^4 [Z(S)] &= [S^{(4)}] - (1 -  \L +  \L^2)[S^{(3)}] -  
   \L \HY [S^{(2)}] + ( \L+ \L^2+ \L^3)[S^2] - \\ &-
   \L^2[S^{(2)}] -  ( \L + 2 \L^3 + \L^5) \HY + ( \L^2 + \L^3 + \L^4 + \L^5 + \L^6).
  \end{split}
\end{equation}
\end{restatable*}
\begin{restatable*}{theorem}{Atwothm}\label{thm:A2}
  Let $S$ be any singular cubic surface with one singularity of type $A_2$. The following relation is a beautiful formula $($i.e. it holds in~$\KGro)$ of degree $4$ for $S$ with the LLSvS variety $Z(S)$ and the variety $A(S)$ of lines passing through the singular point:
  \begin{align}\label{eq:SZSA2}
        \L^4[Z(S)] &= [S^{(4)}] - (1-\L+\L^2)[S^{(3)}] - \L[S][S^{(2)}] + (\L + \L^2 + \L^3)[S^2] -\\
	&- \L^2 [S^{(2)}] - (\L + 2\L^3 + \L^5)[S] + \L^2 + \L^3 +\L^4[A(S)]+ \L^5 + \L^6. \nonumber
\end{align}
\end{restatable*}

In Section \ref{SB obstruction} we use the stable birational realization $\muSB\colon\KGro \rightarrow \Z[SB_{k}]$. 
In other words, we investigate formulae modulo $\L$.
This gives us some obstructions to possible forms of beautiful formulae.
In particular, we will show that the $S$-$Z(S)$ relation (\ref{S-Z(S) formula}) can not be true in the Grothendieck ring of varieties $\KGro$.

In Section \ref{Speculations for cubic fourfolds} we use the categorical realization $\muCat\colon \KGro \rightarrow \KCat$. 
In this realization~$\mu_{cat}(\L)=1$.
Under certain assumptions we obtain a hypothetical form of the $Y$-$Z(Y)$ for a smooth cubic fourfold $Y$:
\begin{restatable}{conjecture}{fourfoldconj}\label{conj:fourfold}
  The following relation holds in the Grothendieck ring of varieties $\KGro$:
\begin{align*}
  \YZYequation .
\end{align*}
\end{restatable}
In Appendix we place character tables for the Weyl group $\W$ of type $E_6$ and some related groups. We also summarize there all the forms of the $S$-$Z(S)$ relation that we obtain.

\subsection*{Acknowledgements}
I am grateful to Sergey Galkin for the formulation of the problem, numerous useful discussions and permanent attention.
I would like to thank Alexander Kuznetsov, Artem Prihodko, Dmitry Kubrak, Galina Ryazanskaya for discussions, corrections, references and their interest in this work.
The author is partially supported by Laboratory of Mirror Symmetry NRU HSE, RF Government grant, ag. \textnumero\ 14.641.31.0001.


\section{Preliminaries} \label{Preliminaries}
In Section~\ref{KGro} we recall the basics on the Grothendieck ring of varieties. 
While this ring is too complicated one may consider realizations from this ring to some simpler rings. 
We list certain realizations which we will need in further sections.
In Section~\ref{sec:burnside} we introduce the Burnside ring of a field. 
This ring will help us later to obtain some relations between classes of zero dimensional schemes in the Grothendieck ring of varieties.
In Section~\ref{The 27 lines} we describe symmetries of $27$ lines on a smooth cubic surface $S$. 
These symmetries can be viewed through the rich structure on $H^2(S,\Z)$.
Finally, in Section~\ref{LLSvS} we recall some results about the LLSvS variety $Z(Y)$ of $\P^2$-families of generalized twisted cubic curves on a cubic hypersurface $Y$.

\subsection{The Grothendieck ring of varieties.}\label{KGro}
The Grothendieck ring of varieties appeared for the first time as the "K-group" in the Grothendieck's letter \cite{Gro} to Serre dated 16 August~1964.
Let $k$ be a field.
The \emph{Grothendieck ring} $\KGro$ of $k$-varieties as an abelian group is generated by the isomorphism classes of reduced, quasi-projective $k$-schemes.
The relations are
$$ [X] = [Y] + [X \setminus Y],$$
whenever $Y$ is a closed subscheme of $X$.

The multiplication in $\KGro$ is defined by $[X]\cdot [Y]$ = $[X\times_{k} Y]$. 
The neutral element is~$1 = [\Spec(k)]$. 
We denote by $\L:=[\A^1_k]$ the class of an affine line.
Detailed references on the Grothendieck ring of varieties are \cite{Loo02, Bit04}. 

Let $E$ over $B$ be a Zariski locally-trivial fibration with the fiber $F$. Then
\begin{equation}\label{eq:fibration}
  [E] = [F]\cdot[B].
\end{equation}
This is proved by induction on dimension of $B$.

Let $X$ to be a smooth variety and $W\subset X$ be a smooth closed subvariety of codimension $c$. Then
\begin{equation}\label{eq:bittner}
  [\Blow_W(X)] - [\P(N_{W/X})] = [X] - [W], 
\end{equation}
where $\Blow_W(X)$ is the blow up of $X$ in $W$ and $N_{W/X}$ is the normal bundle to $W$. 
Note that~$[\P(N_{W/X})]=[\P^{c-1}][W]$ by \eqref{eq:fibration}. 
If $k$ is of characteristic zero, then there is an alternative description of the Grothendieck ring $\KGro$ due to Bittner: the generators are classes of smooth projective connected varieties and the relations are of the form~\eqref{eq:bittner}, \cite{Bit04}.

For us a \emph{pre-$\lambda$-ring} is a commutative ring with an identity element $1$ and a set of operations $\lambda^n$ such that
\begin{enumerate}
        \item $\lambda^0(x) = 1$;
        \item $\lambda^1(x) = x$;
	\item \label{eq:lambdasum} $\lambda^n(x+y) = \sum_{i+j=n} \lambda^i(x)\lambda^j(y)$.
\end{enumerate}

Denote by $\Sym^n(X)=:X^{(n)}$ the $n$-th symmetric powers of $X$. For each $n \geqslant 0$ the operations~$X\mapsto \Sym^n(X)$ descends to $\KGro$ and define the pre-$\lambda$-ring structure on it, i.e.~$\lambda^k([X])=[X^{(n)}]$, \cite{GLM06, GLM13}. 

Consider a power series $\Lambda(x,t) = \sum_{n=0}^{\infty}\lambda^n(x)t^n$.
Then the axioms of a pre-$\lambda$-ring is equivalent to the requirement that $\Lambda(x,t) = 1 + xt + \sum_{n=2}^{\infty} \lambda^n(x)t^n$ and it is multiplicative, i.e 
$$\Lambda(x+y,t) = \Lambda(x,t)\cdot\Lambda(y,t).$$
In these terms the structure of a pre-$\lambda$-ring on $\KGro$ is defined by the multiplicative \emph{Kapranov's zeta function}:
$$Z_{Kap}(X,t) = 1 + [X]t + \sum_{n=2}^{\infty}[X^{(n)}]t^n\in\KGro[[t]].$$
We also have, \hbox{\cite[Lemma~4.4]{Go01}}:
\begin{equation}
\Sym^n(\L^m\alpha) = \L^{nm}\Sym^n(\alpha),\ \alpha \in \KGro.
\end{equation}

The Grothendieck ring of varieties is still very poorly understood. 
We will call a ring morphism $\mu$ from $\KGro$ to any ring $R$ a \emph{realization homomorphism with values in $R$}. Some authors used the term \emph{motivic measure} for a ring homomorphism $\mu:\KGro\rightarrow R$.
See \cite{GS14} for some well-known examples of realization homomorphisms and their application to the study of cubic hypersurfaces. 
Here we list the realization homomorphisms which will be used in this work.

\subsubsection{Stable birational realization}\label{sec:SBreal}
Two smooth projective varieties $X$ and $Y$ are called stably birationally equivalent if for some $n,m \geqslant 1$ the variety $X\times \P^m$ is birationally equivalent to~$Y\times \P^n$.

Let $\Z[SB_k]$ be a free abelian group generated by the stable birational equivalence classes of smooth, projective, irreducible $k$-varieties. One can equip it with multiplication that given by the product of varieties, \cite{LL03}.
In characteristic zero, the quotient of $\KGro$ by the principal ideal $(\L)$ is naturally isomorphic to the ring $\Z[SB_k]$.
In other words there exists a realization homomorphism
$$ \mu_{sb}\colon \KGro \rightarrow \Z[SB_k],$$
which sends the class $[X_k]$ of a smooth projective variety to its stable birational equivalence class.

We will need the following result:
\begin{theorem}\label{mod L theorem}\cite{LL03}
        Let $k$ be a field of characteristic zero. If $X$ and $Y_1, \dots, Y_m$ are smooth projective connected varieties and 
      $$\KGro/\L\ni [X] = \sum_{j=1}^{m}n_j[Y_j] $$
        for some $n_j \in \Z$, then $X$ is stably birationally equivalent to one of the $Y_j$.
\end{theorem}

\begin{remark}\label{rmk:L-sing}
Note that Theorem \ref{mod L theorem} also can be applied to singular $X$ and $Y_i$ if their singularities are \emph{$\L$-singularities}.
We say that $X$ has $\L$-singularities if $[\tilde{X}]\equiv [X] ~(mod~\L)$, where $\tilde{X}$ is a resolution of singularities. 
This definition does not depend on the choice of $\tilde{X}$ since the kernel of the realization $\mu_{sb}$ is the ideal generated by $\L$. 
	
\end{remark}

\subsubsection{Categorical realization}\label{sec:catreal}
The ring $\KCat$ as an abelian group is generated by quasi-equivalence classes of pre-triangulated dg-categories. The relations are
$$[\CC] = [\CA] + [\CB]$$
whenever there is a semi-orthogonal decomposition $\CC = \langle \CA, \CB \rangle$. See Section 4 in \cite{BLL04} for details and the definition of the product on this ring. 
There exists a realization homomorphism
$$\muCat\colon \KGro \rightarrow \KCat$$
which sends the class $[X_k]$ of a smooth projective variety to the class $[\D(X)]$ of its bounded derived category.
Note that $\muCat(\P^d) = d+1$, \cite[Example~7.4]{BLL04}, and so $\muCat(\L)=1$. 

Symmetric power operations $\Sym_{GK}^n$ in sense of Ganter and Kapranov, \cite{GK14}, give rise to $\lambda$-operations on $\KCat$, \cite{GS15}.
This allows to define the \emph{Galkin-Shinder's categorical zeta-function}:
$$ Z_{GS}(\CC, t) = \sum_{n\geqslant 0}^{\infty}[\Sym_{GK}^n\CC]t^n.$$
Let $X$ be a smooth projective variety. The expressions for $\muCat([X^{(n)}])$ follow from the relation~\cite{GS15, BGLL17} between $Z_{Kap}(X, t)$ and $Z_{GS}(X,t)$:
\begin{equation}\label{Zmot=Zcat}
        Z_{GS}(\muCat([X]), t) = \prod_{m\geqslant 1}^{\infty}\muCat(Z_{Kap}(X,t^m)).
\end{equation}
\begin{example}\label{SymCat}
        Comparing coefficients of $t^2$ in \eqref{Zmot=Zcat} we obtain:
        \begin{equation}\label{CatSym}
	  \muCat([X^{(2)}]) = \Sym^2_{GK}(\muCat([X])) - \muCat([X]).
        \end{equation}
        Let $X$ be $m$ isolated points, i.e. $[X] = m$ in $\KGro$. From \eqref{CatSym} it follows that
        \begin{equation}\label{Sym0GK(m)}
	  \Sym^2_{GK}(m) = \binom{m}{2} + 2m.
	\end{equation}
\end{example}

\subsubsection{Gillet--Soul\'e motivic realization}\label{sec:GSreal}
Let $k$ be a field of characteristic zero.
Denote by $\Chow$ the Grothendieck's category of Chow $k$-motives, \cite{Sch94}.
There exists \cite{GS96} a realization homomorphism
$$\muGS\colon \KGro \rightarrow \K_0(\Chow)$$
such that in the case $X$ is smooth and projective, $\muGS([X])=[\h(X)]$ the class of the motive.
This is a homomorphism of pre-$\lambda$-rings since $\muGS([\Sym^n X])=[\Sym^n\h(X)]$, \cite{RA98}, see also~\cite[Theorem~2.2]{Go01} for detailed discussion.

Let $\ChowArt$ be the full subcategory of $\Chow$ generated by the motives of smooth zero-dimensional projective varieties and their summands. 
We will call these motives \emph{Artin motives}.
The motive $\h(X)$ of a smooth zero-dimensional projective variety is completely described by the action of the Galois group $\Gal_k$ on its geometric points.
There exists an equivalence of $\ChowArt$ with the category of Galois representations which sends $\h(X)$ to the \emph{associated Galois module}~$\Q^{X(\bar{k})}$,~\cite{SP11}.
By \emph{a Galois representation} here we understand a rational finite-dimensional representation of the Galois group $\Gal_k$ such that the action of $\Gal_k$ factors through a finite group. 

There exists a motive $\L\in \Chow$, called the \emph{Lefschetz motive}, such that the motive of the projective line decomposes as $\h(\P^1)=\h(\Spec(k))\oplus\L$.
Note that by construction~\cite{Sch94} the motive $\L$ is an invertible motive, i.e. there exists $\L^{-1}$ such that $\L\otimes\L^{-1} \cong \Q:= \h(pt)$.
We will call a motive in $\Chow$ \emph{zero-dimensional} if it is a finite direct sum of $M\L^n:=M\otimes \L^n$ where $M$ is an Artin motive.
Denote by $\ChowZero\subset \Chow$ the full subcategory generated by zero-dimensional motives.
This category is equivalent to the category of graded Galois representations.
 
\subsubsection{G\"ottsche formula}
The following relation \cite{Go01, LL18} in $\KGro$ connects symmetric powers $X^{(n)}$ and Hilbert schemes of points $X^{[n]}$ in the case of a smooth surface $S$. 
\begin{equation}\label{Gottsche-formula}
  \sum_{n=0}^{\infty}[S^{[n]}]t^n = \prod_{i=1}^{\infty}Z_{Kap}(S, \L^{i-1}t^i).
\end{equation}
\begin{example}\label{Hilb-Sym234}
        The coefficients of $t^2, t^3$ and $t^4$ give the following relations
        \begin{align*}
          [S^{[2]}]   &= [S^{(2)}] + \L [S], \\
                [S^{[3]}]   &= [S^{(3)}] + \L [S^2] + \L^2 [S], \\
          [S^{[4]}]   &= [S^{(4)}] + \L [S] [S^{(2)}] + \L^2 [S^{(2)}] + \L^2 [S^2] + \L^3 [S].
\end{align*}

\end{example}

\subsection{Burnside ring}\label{sec:burnside}
Let $G$ be a discrete group. Denote by $\Burn_+(G)$ the set of isomorphism classes of finite $G$-sets.
It is a commutative semi-ring with addition and multiplication given by disjoint union and Cartesian product. 
As an additive monoid it is a free commutative monoid generated by isomorphism classes of finite irreducible $G$-orbits, i.e. conjugacy classes of subgroups of $G$ of finite index.
The Burnside ring $\Burn(G)$ is defined as the associated Grothendieck ring.
It is a combinatorial avatar of the representation ring of $G$.
The pre-$\lambda$-ring structure on $\Burn(G)$ is defined by the operations $\lambda^n(S)=[S^n/\S_n]$, where $S$ is a $G$-set,~\cite{Rok11}.
Denote by $\Burn(k)$ the Burnside ring $\Burn(\Gal(k))$ of the Galois group $\Gal_k$ considered as a discrete group.
See \cite{Rok11} for details about Burnside ring of a field and its connections with the Grothendieck ring of varieties $\KGro$.

Denote by $\Burn_+(k)$ the set of isomorphism classes of smooth zero dimensional schemes of finite length over $k$. 
It is a commutative semi-ring and generated as an additive monoid by isomorphism classes of finite field extensions $K/k$.
The Burnside ring $\Burn(k)$ is naturally identified with the Grothendieck ring of the semi-ring $\Burn_+(k)$.

From the last description of $\Burn(k)$ we see a natural map $\mathrm{Art}_k\colon \Burn(k)\rightarrow \KGro$.
This is a homomorphism of pre-$\lambda$-rings, \cite{Rok11}.
\begin{remark}
  The map $\mathrm{Art}_k\colon\Burn(k)\rightarrow \KGro$ is an embedding for algebraically closed fields, finite fields and fields of characteristic zero \cite{LS10}. 
\end{remark}
\subsection{The 27 lines.}\label{The 27 lines}
Let $S$ be a smooth cubic surface over a perfect field $k$. 
It is a classical result~\cite{Ca49},~\cite{Sal49} that over algebraically closed field $\bar{k}$ there are $27$ lines on $S$. There is a natural action of a certain group on the set of these lines. 
Consider the geometric N\'eron-Severi group $N(S):=\Pic(S\otimes \bar{k})$.
Denote by $K_S\in N(S)$ the canonical class of $S$.
According to~\cite{Ma74} the group $N(S)$ has the following structure:
\begin{enumerate}
        \item The group $N(S)=\Z^7=\overset{6}{\underset{i=0}{\bigoplus}}\Z E_i$.
        \item The class $K_S = -3E_0 + E_1 + \dots + E_6$.
        \item The intersection form $N(S)\times N(S)\rightarrow \Z$ is given by:
	  $$ (E_0,E_0)=1, (E_i,E_j) = -\delta_{ij} \text{ for } j\neq 0.$$
        \item The set $R = \{r\in N(S) \mid (r, K_S)=0, (r,r)=-2\}$ is the root system of the Weyl group~$\W$ of the Lie group $E_6$.
        \item The set $I = \{l\in N(S) \mid (l, K_S)=-1, (l,l)=-1\}$ consists of classes of lines on $S$.
\end{enumerate}

The group which stabilizes $K_S$ and preserves the intersection form is $\W$. This group acts on the set of lines $I$ and on the set of roots $R$.

\begin{remark}
        Geometrically we can see this structure in the following way.
        Over $\bar{k}$ surface~$S$ is a blow up of $\P^2$ in six points. 
        Then $E_1,\dots, E_6$ are the exceptional lines of this blow up and~$E_0$ is the image of the class of a line in $\P^2$.        
\end{remark}

The 27 classes of lines, i.e. vectors from $I$ have the following description in terms of $E_i$.
\begin{itemize}
        \item Six vectors $E_i$, $i\neq 0$.
        \item Fifteen vectors $E_0 - E_i - E_j$, $i\neq j \neq 0$.
	\item Six vectors $2E_0 + E_i - \sum_{j=1}^{6} E_j$, $i\neq 0$.
\end{itemize}

The 72 vectors of the root system $R$ have the following description in terms of $E_i$.
\begin{itemize}
        \item One vector $2E_0 - \sum E_i$.
        \item Twenty vectors $E_0 - E_i - E_j - E_k$, $i\neq j\neq k\neq 0$.
        \item Thirty vectors $E_i - E_j$, $i\neq j$.
        \item Twenty vectors $-E_0 + E_i + E_j + E_k$, $i\neq j\neq k\neq 0$.
	\item One vector $-2E_0 + \sum_{i=1}^{6} E_i$.
\end{itemize}

\subsection{Twisted cubics}\label{LLSvS}
Let $Y$ be a smooth cubic hypersurface.
Let $M_3(Y)$ denote the compactification of the subscheme of twisted cubic curves in the Hilbert scheme $\Hilb^{3n+1}(Y)$, where~$3n+1$ is the Hilbert polynomial in variable $n$.
The scheme $M_3(Y) =: \Hilb^{gtc}(Y)$ is a moduli space of generalized twisted cubic curves on $Y$.
Geometry of $M_3(Y)$ is quite interesting.
\begin{theorem}\label{LLSvS theorem}\cite{LLSvS}
        Let $Y\subset \P^5$ be a smooth cubic hypersurface that does not contain a plane. 
        Then the moduli space $M_3(Y)$ of generalised twisted cubic curves on $Y$ is a smooth and irreducible projective variety of dimension $10$.
	Further, there is a smooth eight-dimensional holomorphically symplectic manifold $Z(Y)$ and morphisms $u\colon M_3(Y) \rightarrow Z(Y)$ and $j\colon~Y\rightarrow~Z(Y)$ with the following properties:
        \begin{enumerate}
                \item The morphism $j$ is a closed embedding of $Y$ as a Lagrangian submanifold in $Z(Y)$.
                \item The morphism $u$ factors as follows:
                        \begin{center}\label{Z contraction} \begin{tikzpicture}
\matrix (m) [matrix of math nodes, row sep=2em,
column sep=2em]
{  M_3(Y) && Z(Y) \\
                & Z' & \\ };
        \path[->] (m-1-1) edge node[auto,swap]{$a$}        (m-2-2);
        \path[->] (m-2-2) edge node[auto,swap]{$\sigma$}   (m-1-3);
        \path[->] (m-1-1) edge node[auto]{$u$}        (m-1-3);
\end{tikzpicture}
\end{center}
where $a\colon M_3(Y) \rightarrow Z'$ is a $\P^2$-fiber bundle and $\sigma\colon Z'\rightarrow Z(Y)$ is the blow-up of $Z(Y)$ along $Y$. 
        \end{enumerate}
\end{theorem}

In the case of cubic surfaces $S$ with at most rational double point singularities generalized twisted cubic curves come in $\P^2$-families. 
Let $\tilde{S}\rightarrow S$ be a minimal resolution of such a singular cubic surface $S$.
The smooth surface $\tilde{S}$ is a weak Del Pezzo surface.
The orthogonal complement $\Lambda := K_{\tilde{S}}^{\perp} \subset H^2(\tilde{S}, \Z)$ of the canonical divisor is a negative definite root lattice of type $E_6$, \cite{LLSvS}. 
The components of the exceptional divisor of this minimal resolution $\tilde{S}$ are $(-2)$-curves whose classes form a subset in the root system $R\subset \Lambda$ that is a root basis for a subsystem $R_0\subset R$. 
The connected components of the Dynkin diagram of $R_0$ are in bijection with the singularities of $S$.
This limits the possible combinations of singularity types of $S$ to the following list:
$A_1 , 2A_1 , A_2 , 3A_1 , A_1 + A_2 , A_3 , 4A_1 , 2A_1 + A_2 , A_1 + A_3 , 2A_2 , A_4 , D_4 , 2A_1 + A_3 ,
A_1 + 2A_2 , A_5 , D_5 , A_1 + A_5 , 3A_2 , E_6$, see \cite{Dol12} for details.

\begin{theorem}\cite{LLSvS}\label{gtc on sing surf}
    Let $S$ be a cubic surface with at most rational double point singularities.
    Then 
    $$ Hilb^{gtc}(S)_{red} \cong (R/\W(R_0))\times \P^2.$$
\end{theorem}
That is why the set geometric points of $Z(S)$ is $R/\W(R_0)$.

\addtocontents{toc}{\protect\setcounter{tocdepth}{2}}

\section{When $\L = 0$. Obstructions from stable birational geometry}
\label{SB obstruction}
In this section we are going to use the \emph{stable birational realization}, Section~\ref{sec:SBreal}:
$$\muSB\colon\KGro \rightarrow \Z[SB_k].$$
The kernel of this map is the ideal generated by $\L$, \cite{LL03}. 
In other words we will investigate beautiful formulae modulo $\L$.
This realization allows us to distinguish elements in $\KGro$ using stable birational invariants.
The main tool to pass some information in the other direction is Theorem~\ref{mod L theorem} by Larsen and Lunts.

The existence of a $k$-point is a stable birational invariant. 
In \cite{Nis55} it is proved that if a regular $k$-scheme has a $k$-point then any proper $k$-scheme birational to it also has a $k$-point.
Around 1992 Endre Szab\'o found a new short argument.
The proof is reproduced in~\cite[Proposition~A.6]{KS00} and in \cite[page~183]{KSC04}.

Let $k$ be any field and $Y$ be any cubic hypersurface over $k$.
There exists always a $k$-point on the third symmetric power $Y^{(3)}$.
Indeed, pass any $k$-line. 
If the intersection with $Y$ is a line than $Y$ has a $k$-point.  
Otherwise, the intersection is a zero dimensional scheme.
If this scheme has no $k$-points then $Y$ has a $K$-point for some field extension of degree~$3$. 
This $K$-point gives a $k$-point on $Y^{(3)}$.
In the other case $Y$ already has a $k$-point and so $Y^{(3)}$, too.

\begin{lemma}\label{SBnonequiv}
Let $n,m$ be natural numbers and assume that $n$ is prime to $3$.
Then there exists a smooth cubic surface $S$ such that $S^{(n)}$ is not stably birationally equivalent to $S^{(3m)}$.
Moreover~$S^{(n_1)}\times\dots\times S^{(n_k)}$ is not stably birationally equivalent to $S^{(3m)}$ if some of $n_i$ is prime~to~$3$.
\end{lemma}
\begin{proof}
There exist smooth cubic surfaces which have no points over any field extension $K/k$ of degree prime to $3$. 
For example, Lemma 3.4. in \cite{CM04}, consider a cubic surface $S$ over $k=\Q_p$ defined by an equation $x_0^3 + px_1^3 + p^2x_2^3 - ax_3^3$, where $a$ is not a cube modulo $p$.
Comparing valuations, it is easy to see that $S$ has no $\Q_p$-points. 
Over local fields Cassels and Swinnerton-Dyer Conjecture~\ref{CSDconj} is a theorem.
That is why this cubic surface $S$ has no points over field extensions $K/k$ of degree prime to $3$.
It follows that symmetric powers $S^{(n)}$ can not have a~$k$-point when $n$ prime to $3$.
Other symmetric powers $S^{(3n)}$ always have a $k$-point.

Since the existence of a $k$-point is a stable birational invariant, as discussed above,
$S^{(n)}$ is not stably birationally equivalent to $S^{(3m)}$. The scheme $S^{(n_1)}\times\dots\times S^{(n_k)}$ also does not have a $k$-point if some $n_i$ is prime to $3$. 
\end{proof}
\begin{theorem}\label{SBobsth}
Suppose a formula in $\KGro$ for a smooth cubic surface $S$ with some varieties modulo $\L$ has the following form
\begin{equation}\label{eq:false}
  [S^{(3m)}] \equiv \sum_i [S^{(n_1^i)}\times\dots\times S^{(n_{k_i}^i)}]~(mod~\L),
\end{equation}
where for each $i$ some of $n^i_j$ is prime to $3$.
Then this formula is not beautiful, i.e. it can not be true in $\KGro$ for any smooth cubic surface over any field $k$ (may be except some fields of small characteristic).
\end{theorem}
\begin{proof}
The Hilbert schemes of points on a smooth surface $S$ are smooth. The G\"ottsche formula~\eqref{Gottsche-formula} shows that singularities of symmetric powers $S^{(n)}$ are $\L$-singularities and by Remark~\ref{rmk:L-sing} we can apply Theorem \ref{mod L theorem} to them. 
Consider the cubic surface $S$ from Lemma~\ref{SBnonequiv} then the equation \eqref{eq:false} contradicts to Theorem~\ref{mod L theorem}.
\end{proof}
\begin{corollary}
  The $S$-$Z(S)$ relation \eqref{S-Z(S) formula}, which we will obtain in Section~\ref{sec:smooth} below, is not a beautiful formula in $\KGro$ for smooth cubics $S$ with the LLSvS variety $Z(S)$.
\begin{proof}
 The $S$-$Z(S)$ relation (\ref{S-Z(S) formula}) implies $[S^{(4)}] \equiv [S^{(3)}]~(mod~\L)$.
\end{proof}

\end{corollary}
\begin{example}\label{nottrue}
  The relation (\ref{Homrel5}) from Section~\ref{sec:smooth} would imply
        $$[S^{(5)}] + [S\times S^{(3)}] \equiv [S\times S^{(4)}] + [S^{(3)}]~(mod~\L).$$
By Theorem~\ref{SBobsth} this relation is not a homogeneous beautiful formula in $\KGro$ for smooth cubics $S$.
\end{example}

In the case of singular cubic surfaces the situation with the existence of $k$-points is a bit different. For example in the case of singularity of type $A_1$ on $S$ always exists a $k$-point. 
\begin{proposition}
  Let $S$ be a singular cubic surface with one rational double point singularity of type $A_1$ such that~$\Delta = 0$, see Lemma~\ref{lem:SA1Gro} over a field of characteristic zero. Then~$S$,~$S^{(2)}$,~$S^{(3)}$ and $S^{(4)}$ are stably birationally equivalent.
\end{proposition}
\begin{proof}  
  The relation \eqref{eq:SZS3A1} from Section~\ref{Singular cubic surfaces} in $\KGro$ implies:
\begin{equation}\label{eq:S3-S2-A1}
        [S^{(3)}]\cong [S^{(2)}]~(mod~\L).
\end{equation}
The relation (\ref{eq:SS4A1}) in $\KGro$ implies:
\begin{equation}\label{eq:S4-S3-A1}
        [S^{(4)}]\cong [S^{(3)}]~(mod~\L).
\end{equation}

Consider a resolution of singularities $\tilde{S}$ of $S$.
In $\KGro$ we have $[\tilde{S}]=[S] + \L$.
It follows that~$[S^{(n)}] \equiv [\tilde{S}^{(n)}]~(mod~\L)$. 
This allows to apply Theorem~\ref{mod L theorem} to \eqref{eq:S3-S2-A1}, \eqref{eq:S4-S3-A1}.
Using also \hbox{the~$Y$-$F(Y)$} relation (\ref{The Y-F(Y) relation Sym}) we see that
$S$, $S^{(2)}$, $S^{(3)}$ and $S^{(4)}$ are stably birationally equivalent.
\end{proof}

\section{When $\L = 1$. Categorical heuristics} \label{Speculations for cubic fourfolds}
In this section we consider the \emph{categorical realization homomorphism}, Section~\ref{sec:catreal}:
$$\muCat\colon \KGro \rightarrow \KCat.$$
It sends the class $[X_k]$ of a smooth projective variety to the class $[\D(X)]$ of the bounded derived category of coherent sheaves on $X$.
Note that $\muCat(\L)=1$.
A semi-orthogonal decomposition of any category $\CC$ gives a relation in $\KCat$. 
A general heuristic is that such a relation can be lifted to a relation in the Grothendieck ring of varieties $\KGro$.
Using this for cubic fourfolds we will obtain some hypothetical form of the $Y$-$Z(Y)$ relation.

Let $Y\subset \P^5$ be a smooth cubic fourfold. 
There are deep connections between cubic fourfolds and K3 surfaces.
First of all, on the level of Hodge diamonds we have the following decomposition: 
\renewcommand{\arraystretch}{0.5}
\newcolumntype{x}[1]{>{\centering}p{#1}}
\newcommand{\ms}{10pt}
\begin{equation*}
        \begin{array}{x{\ms}@{}x{\ms}@{}x{\ms}@{}x{\ms}@{}x{\ms}@{}x{\ms}@{}x{\ms}@{}x{\ms}@{}x{\ms}@{}p{0pt}}
                     &&&& 1 &&&&        &\\
                    &&& 0 && 0 &&&      &\\
                  && 0 && 1 && 0 &&     &\\
                 & 0 && 0 && 0 && 0 &   &\\
                0 && 1 && 21 && 1 && 0  &\\
                 & 0 && 0 && 0 && 0 &   &\\
                  && 0 && 1 && 0 &&     &\\
                    &&& 0 && 0 &&&      &\\
                     &&&& 1 &&&&        &
        \end{array} 
        = 
        \begin{array}{x{\ms}@{}x{\ms}@{}x{\ms}@{}x{\ms}@{}x{\ms}@{}p{0pt}}
             && 1 &&       &\\
            & 0 && 0 &     &\\
           1 && 20 && 1    &\\
            & 0 && 0 &     &\\
             && 1 &&       & 
        \end{array}
        +
        \begin{array}{x{\ms}@{}x{\ms}@{}x{\ms}@{}x{\ms}@{}x{\ms}@{}x{\ms}@{}x{\ms}@{}x{\ms}@{}x{\ms}@{}p{0pt}}
                     &&&& 1 &&&&        &\\
                    &&& 0 && 0 &&&      &\\
                  && 0 && 0 && 0 &&     &\\
                 & 0 && 0 && 0 && 0 &   &\\
	        0 && 0 && 1  && 0 && 0~,&\\
                 & 0 && 0 && 0 && 0 &   &\\
                  && 0 && 0 && 0 &&     &\\
                    &&& 0 && 0 &&&      &\\
                     &&&& 1 &&&&        &
        \end{array}
\end{equation*}
where the small Hodge diamond is the Hodge diamond of a K3 surface. 

Moreover, smooth cubic fourfolds are linked to K3 surfaces via their Hodge structures. 
The Hodge structure $H^2(\Kthree,\C)^{prim}(-1)$ is isomorphic to $K_d^{\perp}\subset H^4(Y, \C)$ for some special labelled cubic fourfolds $(Y,K_d)$ and some polarized K3 surfaces $\Kthree$, see \cite{Ha00}.   
In this case, it is said that the K3 surface is \emph{associated} to the special cubic fourfold.

There exist also some connections between K3 surfaces and moduli spaces of rational curves on cubic fourfolds:
\begin{theorem}\label{Hilb2K3} \cite{BD85, Ha00}
        Let $Y$ be a smooth cubic hypersurface in $\P^5_{\C}$. The Fano variety~$F(Y)$ of lines on $Y$ is a smooth four-dimensional holomorphically symplectic variety which is deformation equivalent to the second Hilbert scheme of a K3 surface.
        If $Y$ is a Pfaffian cubic fourfold not containing a plane with associated K3 surface $\Kthree$ not containing a line then~$F(Y)\cong \Hilb^2(\Kthree)$.
\end{theorem} 
\begin{theorem}\label{Hilb4K3} \cite{AL14},\cite{Le15}
        Let $Y\subset\P^5_{\C}$ be a smooth cubic hypersurface that does not contain a plane.
	The LLSvS variety $Z(Y)$, Theorem~\ref{LLSvS theorem}, is a smooth eight-dimensional holomorphically symplectic variety which is deformation equivalent to the fourth Hilbert scheme of a K3 surface. In the case of a Pfaffian cubic fourfold not containing a plane with the associated K3 surface~$\Kthree$ not containing a line $Z(Y)$ is birational to $\Hilb^4(\Kthree)$.
\end{theorem}

There exists a semiorthogonal decomposition, \cite{Kuz08}:
$$ \mathcal{D}(Y) = \langle \mathcal{A}_Y, \CO_Y, \CO_Y(1), \CO_Y(2) \rangle.$$
By many features the category $\CA_Y$ looks like the derived category of a K3 surface and for some cubic fourfolds $\CA_Y$ is equivalent to the derived category of a K3 surface, \cite{Kuz08}.

In Example~\ref{SymCat} we recalled how to express $\muCat([X^{(n)}])$ through the classes of symmetric powers of $\muCat([X])\cong [\D(X)]$.
Let us evaluate the $Y$-$F(Y)$ relation (\ref{The Y-F(Y) relation Sym}) in $\KCat$:
$$ \muCat([Y^{(2)}]) = \Sym^2_{GK}([\D(Y)])-[\D(Y)] = 2 [\D(Y)] + [\D(F(Y))],$$
$$ [\D(Y)] = [\CA_Y] + 3,$$
$$ \Sym^2_{GK}([\CA_Y] + 3) - ([\CA_Y] + 3) = 2([\CA_Y] + 3) + [\D(F(Y))].$$
We can expand the symmetric square using Axiom~\eqref{eq:lambdasum} in the definition of a pre-$\lambda$-structure. It follows that
$$ [\Sym^2_{GK}\CA_Y] = [\D(F(Y))].$$ 

\begin{conjecture}\emph{[Sergey Galkin]}\label{conj:GS}
        Let $Y$ be a smooth cubic fourfold. 
        \begin{enumerate}
                \item The derived category $\D(F(Y))$ is equivalent to $\Sym^2_{GK}\CA_Y$.
                \item The derived category $\D(Z(Y))$ is equivalent to $\Sym^4_{GK}\CA_Y$. 
        \end{enumerate}
\end{conjecture}

The first conjecture for certain special cubics (Pfaffian cubics) follows from Theorem~\ref{Hilb2K3}.
In~\cite{Kuz06} it was shown that $\CA_Y \cong \D(\Kthree)$ in the case of a Pfaffian cubic.
The derived category $\D(\Hilb^n(\Kthree))$ is equivalent to the derived category $\D_{\S_n}(\Kthree^n)$ of $\S_n$-equivariant coherent sheaves on $\Kthree^n$ by the derived McKay correspondence \cite{Hai99}, \cite{BKR01}, \cite{Kru17}.
And the later category is equivalent to $\Sym^n_{GK}(\D(\Kthree))$, \cite{GK14}.
The second part of the conjecture is motivated by Theorem~\ref{Hilb4K3}. 

\begin{remark}
        In the talk \cite{Gal17} Sergey Galkin explained the strategy to attack this conjecture through gauge linear sigma models.
\end{remark}

Theorems~\ref{Hilb2K3},~\ref{Hilb4K3} and Conjecture~\ref{conj:GS} motivate the following assumption.
\begin{assumption}\label{asm:K3}
  Let $Y$ be a smooth cubic fourfold. Assume that in $\KGro$ we have 
  \begin{align*}
    [Y]    &= \L[\Kthree] + 1 + \L^2 + \L^4, \\
    [F(Y)] &= [\Kthree^{[2]}], \\
    [Z(Y)] &= [\Kthree^{[4]}], 
  \end{align*}
  where $\Kthree$ is some $K3$-surface.
\end{assumption}

Under this assumption we are going to investigate possible forms of a beautiful formula of degree $4$ for smooth cubic fourfolds $Y$ with the LLSvS variety $Z(Y)$.
We will see that \hbox{the $Y$-$F(Y)$} relation (\ref{The Y-F(Y) relation Sym}) also follows from this assumption. 

In order to obtain some relations we are going to express all the related classes through the classes of symmetric powers of $\Kthree$.
We can express the classes $[F(Y)]$ and $[Z(Y)]$ through $[\Kthree]$ and its symmetric powers using G\"ottsche formula (\ref{Gottsche-formula}) for surfaces in the cases of $n=2,4$: 
\begin{align}
  [F(Y)] = [\Kthree^{[2]}]   &= [\Kthree^{(2)}] + \L [\Kthree], \\
  [Z(Y)] = [\Kthree^{[4]}]   &= [\Kthree^{(4)}] + \L [\Kthree] [\Kthree^{(2)}] + \L^2 [\Kthree^{(2)}] + \L^2 [\Kthree^2] + \L^3 [\Kthree]. \label{eq:ZYK3}
\end{align}

The symmetric square of $[Y]$ has the following expression through $[\Kthree]$ and its symmetric powers. We use Axiom~\eqref{eq:lambdasum} from the definition of a pre-$\lambda$-structure. 
\begin{equation}\label{eq:Y2K3}
  [Y^{(2)}]   = \L^2[\Kthree^{(2)}] + (\L^5 + \L^3 + \L)[\Kthree] + \L^8 + \L^6 + 2\L^4 + \L^2 + 1 
\end{equation}
From these expressions we see that 
$$ [Y^{(2)}] = (1 + \L^4)[Y] + \L^2[F(Y)]. $$
This equation coincides with the $Y$-$F(Y)$ relation \eqref{The Y-F(Y) relation Sym} for fourfolds.

\begin{proposition}[Hypothetical $Y$-$Z(Y)$ relation]
	Assume Assumption~\ref{asm:K3} holds. Then in the Grothendieck ring of varieties $\KGro$ we have
\begin{align}\label{eq:YZY}
  \YZYequation.
\end{align}
\end{proposition}
\begin{proof}
Using Axiom~\eqref{eq:lambdasum} from the definition of a pre-$\lambda$-structure we compute: 
\begin{align*}
  [Y^2]       &= \L^2[\Kthree^2] + (2\L^5 + 2\L^3 + 2\L)[\Kthree] + (\L^8 + 2\L^6 + 3\L^4 + 2\L^2 + 1), \\
                [Y^{(3)}]   &= \L^3[\Kthree^{(3)}] + (\L^6 + \L^4 + \L^2)[\Kthree^{(2)}] + (\L^9 + \L^7 + 2\L^5 + \L^3 + \L)[\Kthree] + \\ & 
		+ (\L^{12} + \L^{10} + 2\L^8 + 2\L^6 + 2\L^4 + \L^2 + 1), \\
                [Y][Y^{(2)}]&=  \L^3[\Kthree][\Kthree^{(2)}]  +  (\L^6 + \L^4 + \L^2)[\Kthree^{(2)}] +  (\L^6 + \L^4 + \L^2)[\Kthree^2]+ \\ & 
		+ (2\L^9 + 3\L^7 + 5\L^5 + 3\L^3 + 2\L)[\Kthree] + (\L^{12} + 2\L^{10} + 4\L^8 + 4\L^6 + 4\L^4 + 2\L^2 + 1), \\
                [Y^{(4)}]   &= \L^4[\Kthree^{(4)}] + (\L^7 + \L^5 + \L^3)[\Kthree^{(3)}] + (\L^{10} + \L^8 + 2\L^6 + \L^4 + \L^2)[\Kthree^{(2)}] + \\ & 
		+ (\L^{13} + \L^{11} + 2\L^9 + 2\L^7 + 2\L^5 + \L^3 + \L)[\Kthree] + \\ & 
			    + (\L^{16} + \L^{14} + 2\L^{12} + 2\L^{10} + 3\L^8 + 2\L^6 + 2\L^4 + \L^2 + 1).  
\end{align*}
These equations together with \eqref{eq:Y2K3} and \eqref{eq:ZYK3} imply the proposition.
Note also that easy linear algebra calculations can show that \eqref{eq:YZY} is a unique (up to multiplication) relation between these classes under an assumption that there are no relations between the class of K3 surface and the classes of its symmetric powers in $\KGro$.
\end{proof}

\section{When $\L$ is invertible. Conditions from representation theory} \label{The S-Z(S) relation}
In this section we are going to use the \emph{Gillet-Soul\'e motivic realization homomorphism} (Section~\ref{sec:GSreal})
$$\muGS\colon \KGro \rightarrow \K_0(\Chow)$$
to obtain possible forms of the $S$-$Z(S)$ relation --- a beautiful (Definition~\ref{def:beautiful}) formula of degree~$4$ for cubic surfaces $S$ with the LLSvS variety $Z(S)$ (Section~\ref{LLSvS}). 
In Section~\ref{sec:smooth} we obtain a relation for smooth cubic surfaces in $\K_0(\Chow)$, Theorem~\ref{thm:main}.
In Section~\ref{Singular cubic surfaces} we obtain analogs of the $S$-$Z(S)$ relation for some singular cubic surfaces with one singular point and prove them in $\KGro$.

\subsection{Smooth cubic surfaces}\label{sec:smooth}
Over an algebraically closed field any smooth cubic surface $S$ can be obtained from $\P^2$ by blowing up six points.
So $[S] = 1 + 7\L + \L^2$ in $\KGro$.
Associated moduli spaces $S^{(n)}$, $F(Y)$, $Z(Y)$ can be expressed as polynomials in $\L$.
There are many relations between them.
Fortunately, there is ``more structure'' on the class $[S]$ over other fields.

There exists a rich structure on the geometric N\'eron-Severi group $N(S):= \Pic(S\otimes\bar{k})$ with a big group $\W$ of symmetries, Section~\ref{The 27 lines}.
The group $\W$ is the Weyl group of type $E_6$.
The Galois group $\Gal_k$ also acts on $N(S)$, preserves $K_S$ and the intersection form and so maps to~$\W$. 
Denote by $\W_0\subset\W$ the image of this map.
Denote by $V$ the Artin motive with the associated Galois $\Q$-module $N(S)\otimes \Q$. 
It is known, \cite{Ser91}, \cite[Proposition~14.2.1]{KMP07} that the motive of a smooth cubic surface $S$ is \emph{zero-dimensional}, i.e. it is a sum of the motives of the form~$M\L^n:=M\otimes \L$, where $M$ is an Artin motive and moreover
\begin{equation}\label{eq:Smotive}
\h(S) \cong \Q \oplus V\L \oplus \L^2.
\end{equation}

\begin{lemma}\label{lem:subring}
Let $G$ be a discrete group. Denote by $\Rep(G,\Q)$ the ring of graded rational finite-dimensional representations of $G$ such that the action of $G$ factors through a finite group.
  \begin{enumerate}
    \item The subring in $\K_0(\Chow)$ of zero-dimensional motives is isomorphic to $\Rep(\Gal_k,\Q)$.
    \item Let $S$ be a smooth cubic surface over a field of characteristic $0$. Let $\W_0$ be the image of the $\Gal_k$ in $\W$ described above. Then the motives of $S^{(n)}$, $F(S)$, $Z(S)$ lie in the subring~$\Rep(\W_0,\Q)\subset \Rep(\Gal_k,\Q)\subset \K_0(\Chow)$. The inclusion of~$\Rep(\W_0,\Q)$ in~$\Rep(\Gal_k,\Q)$ is induced by the surjection $\Gal_k\twoheadrightarrow\W_0$. 
  \end{enumerate}
\end{lemma}
\begin{proof}
  (1) The subcategory $\ChowArt$ of Artin motives is equivalent to the category of rational Galois representations, \cite{SP11}, see also Section~\ref{sec:GSreal}. 
  The subcategory $\ChowZero$ of zero-dimensional motives is equivalent to the category of graded rational Galois representations and it is semi-simple.
  The inclusion~$\K_0(\ChowZero)\subset\K_0(\Mot_{num})$ factors through $\K_0(\ChowZero)\rightarrow \K_0(\Chow)$, \cite{MNP13}. Note that the category~$\Mot_{num}$ of motives defined with numerical equivalence is semi-simple.

  (2) The motives of the Fano variety $F(S)$ and the LLSvS variety $Z(S)$ are Artin motives since these varieties are zero-dimensional. 
The set of geometric points of $F(S)$ are in natural bijection with the set of classes of lines $I\subset N(S)$, Section~\ref{The 27 lines}.
The set of geometric points of~$Z(S)$ are in natural bijection with the set of roots $R\subset N(S)$, Section~\ref{LLSvS}.
The motives of the symmetric powers $S^{(n)}$ are also zero-dimensional since $\muGS([\Sym^n X]) = \Sym^n\h(X)$ in~$\Chow$,~\cite{RA98}. 
The action of the Galois group~$\Gal_k$ on the associated Galois modules factors through $\W_0$. 
\end{proof}

\begin{remark}\label{rk:full}
  In \cite[Section~6]{EJ09} the authors provide a concrete example of the smooth cubic surface over $\Q$ defined by the equation
  \[x^3 + 2xy^2 + 11y^3 + 3xz^2 + 5y^2w + 7zw^2 = 0,\]
  such that the image of the Galois group in $\W$ is equal to the whole group $\W$.
\end{remark}

\begin{remark}\label{rk:QeqC}
Representation theory for $\W$ over $\Q$ is the same as representation theory over~$\C$.
Table~\ref{E6} shows that the characters of irreducible representations are rational. 
This implies~\cite{CR62} that for any irreducible $\C$-representation $W$ the representation $W^{\oplus m(W)}$ is defined over $\Q$ where $m(W)$ is the Schur index of $W$.
In \cite{Ben69} it is shown that in the case of the Weyl group $\W$ of type $E_6$ these indices are equal to $1$. 
\end{remark}

\begin{lemma}\label{lem:evaluation}
  Any formula $P$ for smooth cubic surfaces $S$ with the Fano variety $F(S)$ and the LLSvS variety $Z(S)$ can be considered, using Gillet-Soul\'e realization $\mu_{GS}$, as a formula $\tilde{P}$ in the ring $\Rep(\W,\C)$ of graded complex representations of the Weyl group $\W$ of type $E_6$. 
  Moreover, the formula $\tilde{P}$ holds in $\Rep(\W,\C)$ if the formula $P$ is beautiful.
\end{lemma}
\begin{proof}
  Let $S_k$ be a smooth cubic surface from Remark~\ref{rk:full}. 
  Using $S$ consider $P$ as a formula in $\KGro$. 
  By Lemma~\ref{lem:subring} Gillet-Soul\'e realization $\mu_{GS}$ of this formula $P$ lies in the subring~$\Rep(\W,\Q)$. 
  By Remark~\ref{rk:QeqC} representation theory of $\W$ over $\Q$ is equivalent to representation theory over $\C$. 
  So the ring $\Rep(\W,\Q)$ is naturally isomorphic to the ring~$\Rep(\W,\C)$.
  We obtain a formula $\tilde{P}$ in $\Rep(\W,\C)$.

  Recall that if $P$ is a beautiful formula, then it is a formula that holds in any realization for any smooth cubic surface $S$ over any field $k$ (maybe except some fields of small characteristic). 
  That is why the formula $\tilde{P}$ holds in $\Rep(\W,\C)$.
\end{proof}

Now we are going to evaluate classes of the varieties $S^{(n)}, S^{[n]}, F(S), Z(S)$ in $\Rep(\W,\C)$ more explicitly.
By abuse of notation we will denote the class of a variety $X$ in $\Rep(\W,\C)$ as $[X]$. 
An element of the ring $\Rep(\W,\C)$ is completely described by its decomposition in irreducible representations. 
Denote by $\chi_n $ irreducible representations of $\W$ as in the character table (Table~\ref{E6}) in Appendix. 

\begin{lemma}\label{lem:Virr}
  The representation $[V]\in\Rep(\W,\C)\cong \Rep(\W,\Q)$ corresponding to the Artin motive $V$ from \eqref{eq:Smotive} has the following decomposition:
  \[ [V] = 1 + \X_3. \]
\end{lemma}
\begin{proof} 
Recall that $[V]=N(S)\otimes \C$.
In notations of Section~\ref{The 27 lines} consider the action of the group~$\S_6$ on $N(S)$ by permuting $E_i$, $i\neq 0$.
The symmetries of the lattice $K_S^{\perp} \subset N(S)$ is~$\W$.
The canonical class $K_S = -3E_0 + E_1 + \dots + E_6$ is stable under the action of $\S_6$.
So the group~$\S_6$ is a subgroup of $\W$.
To identify the representation $[V]$ it is enough to calculate traces of elements from $\S_6$.

More precisely, $[V]=W+1$ since the canonical class is fixed by the action of $\W$. 
Note that~$W=1+M$, where $M$ is the tautological irreducible representation of $\S_6$. That is why~$W$ is not a sum of one-dimensional representations as $\S_6$-representation.  
The only irreducible representations of $\W$ with dimension $\leqslant 6$ are~$\chi_1,\dots,\chi_4$, Table~\ref{E6}. The representations $\chi_1$ and~$\chi_2$ are one-dimensional and the representations $\chi_3$ and $\chi_4$ are six-dimensional.
This implies that~$W$ is~$\chi_3$ or $\chi_4$.
The representations $W$ is permutational in the basis $\{E_1,\dots,E_6\}$. 
The trace of an element $g\in \S_6$ is equal to the number of fixed vectors in this basis.  
The trace of the action of a simple transposition on $W$ is $4$.
A simple transposition lies in some conjugacy class of $\W$ and only $\chi_3$ has trace equal to $4$ for some conjugacy class.
\end{proof}
\begin{lemma}\label{lem:deg2irr}
  The classes $[S], [S^2], [S^{(2)}], [S^{[2]}]$ and $[F(S)]$ have the following irreducible decompositions in $\Rep(\W,\C)$:
\begin{align}
  [S] &= 1 + ( 1+\chi_3 )\L + \L^2,  \label{eq:Sirr}\\
  [S^2] &= 1 + (2 + 2\X_3 )\L^1 + (4 + 2\X_3 + \X_9 + \X_10 )\L^2 + (2 + 2\X_3 )\L^3 +\L^4, \label{eq:S2irr} \\
  [S^{(2)}] &= \One + (\One + \X_3 )\L + (3 + \X_3 + \X_10 )\L^2 + (\One + \X_3 )\L^3 + \L^4, \label{eq:S2Symirr} \\ 
  [S^{[2]}] &= 1 + (2 + \chi_3 ) \L^1 + (4 + 2\chi_3 + \chi_{10} ) \L^2 + (2 + \X_3 ) \L^3 + \L^4,  \label{eq:SHilb2irr} \\
  [F(S)] &= 1 + \chi_3 + \chi_{10}. \label{eq:Firr} 
\end{align}
\end{lemma}
\begin{proof}
  Equation \eqref{eq:Sirr} follows from Lemma~\ref{lem:Virr} and \eqref{eq:Smotive}.
  Equation~\eqref{eq:S2irr} is the direct multiplication of characters.

  The information in Table~\ref{E6} is enough to calculate symmetric powers of characters up to degree~$5$. 
  Note that the lines $3$--$5$ shows the squares, cubes and fifth powers of the conjugacy classes of~$\W$. 
  Equation \eqref{eq:SHilb2irr} is an application of the G\"ottsche formula~\eqref{Gottsche-formula} to \eqref{eq:S2Symirr}.

  Equation \eqref{eq:Firr} can be obtained in many ways. 
  One can use the same strategy as in the proof of Lemma~\ref{lem:Virr}.
  Let us use the $Y$-$F(Y)$ relation \eqref{The Y-F(Y) relation Sym} instead.
  It is a beautiful formula so by Lemma~\ref{lem:evaluation} we obtain a relation in $\Rep(\W,\C)$.
  Using decompositions \eqref{eq:S2Symirr}, \eqref{eq:Sirr} and considering the coefficients of $\L^2$ we obtain:
  \[ 3 + \chi_3 + \X_10 = 2 + [F(S)]. \]
\end{proof}

\begin{remark}\label{rmk:chi10}
  There exists a map from the representation $[F(S)]$ to the representation $[V]$.
  It sends the class of a line to the corresponding vector in $V\cong N(S)\otimes \C$.
  Since the classes of lines span the whole space $V$ this map is surjective.
  Lemma~\ref{lem:deg2irr} shows that the kernel is $\chi_{10}$.
\end{remark}

\begin{corollary} See Section~\ref{sec:beau} for definitions of homogeneous and beautiful formulae. 
  \begin{enumerate}
    \item There is no homogeneous beautiful formula of degree $2$ for smooth cubic surfaces $S$.
    \item The $Y$-$F(Y)$ relation \eqref{The Y-F(Y) relation} is the unique (up to multiplication) beautiful formula of degree $2$ for smooth cubic surfaces $S$ with the Fano variety $F(S)$.
  \end{enumerate}
\end{corollary}
\begin{proof}
  (1) The classes $[S]$, $[S^2]$ and $[S^{[2]}]$ are the only classes of degree $\leqslant2$ (we use the notation of degree from the definition of beautiful formulae, Section~\ref{sec:beau}). 
  Any homogeneous beautiful formula of degree $2$ for smooth cubic surfaces $S$ would imply a relation (with coefficients in~$\C[\L]$) between these classes in $\Rep(\W,\C)$. 
  The decomposition \eqref{eq:Sirr} of the class~$[S]\in\Rep(\W,\C)$ does not have irreducible summands $\chi_{9}$ and $\chi_{10}$. 
  The decomposition \eqref{eq:SHilb2irr} has $\chi_{10}$. 
  The decomposition \eqref{eq:S2irr} has $\chi_{9}$ and $\chi_{10}$.
  So there are no relations (with coefficients in $\C[\L]$) between them in~$\Rep(\W, \C)$.

  (2) 
  Let us call a relation \emph{minimal} if it is not obtained as a multiplication of another relation by some non-invertible polynomial in $\L$. 
  The part (1) implies that there is only one (if exists) minimal relation (with coefficients in $\C[\L]$) between $[F(S)]$ and other classes. 
  The coefficient between~$[F(S)]$ is divided by $\L^2$ in any relation with $[F(S)]$, since $\chi_{10}$ appears only before $\L^2$ in the irreducible decompositions of $[S^2]$ and $[S^{(2)}]$.
    That is why the $Y$-$F(Y)$ relation is the minimal relation.
\end{proof}

We are going to apply the same strategy for searching beautiful formulae of degree $3$ and $4$ for smooth cubic surfaces $S$ with the LLSvS variety $Z(S)$.
The representation $[Z(S)]$ is obtained from the action of $\W$ on the $72$ $\P^2$-families of generalized twisted cubics on $S$. These families are in bijection with roots of the root system of type $E_6$, Theorem~\ref{gtc on sing surf}.
\begin{lemma}\label{lem:Zirr}
  The 72-dimensional representation $[Z(S)]\in\Rep(\W,\C)$ has the following decomposition in irreducible representations:
  \begin{equation}\label{eq:Zirr}
    [Z(S)]= 1 + \X_3 + \X_8 + \X_10 + \X_16.
  \end{equation}
\end{lemma}
\begin{proof}
  We obtained this using explicit description of roots, Section~\ref{The 27 lines}. 
  Using \texttt{SageMath} we obtained explicit representatives for each conjugacy class of $\W$ in terms of simple reflections and calculated traces explicitly.
\end{proof}
\begin{remark}

  This decomposition can be seen in the following way.
  For each root $\alpha \in R$, the root $-\alpha\in R$. 
  This gives a decomposition $[Z(S)]=M_+ \oplus M_-$, where $M_+$ is generated by sums of opposite roots and $M_-$ by differences of opposite roots.
  
  There is a map from $[Z(S)]$ to $V$ which sends a root to the corresponding vector in $V$. 
  The image is a six-dimensional subspace.
  The subspace $M_+$ lies in a kernel. 
  That is why we can consider this map as a map from $M_-$.
  Denote by $V_{30}$ its $30$-dimensional kernel. 

  According to \cite{Dol12} the set of roots are in natural ($\W$-equivariant) bijection with the sixes of orthogonal classes of lines. 
  This defines a map from $M_+$ to $[F(S)]$.
  It sends a root to the sum of six corresponding classes of lines.
  The image is $\chi_{10}$ defined as a subspace in $[F(S)]$ in Remark~\ref{rmk:chi10}.
  The kernel is a sum of a trivial representation and some $15$-dimensional representation $V_{15}$.

  Using \eqref{eq:Zirr} we obtain that $V_{15}$ is isomorphic to $\chi_8$ and $V_{30}$ is isomorphic to $\chi_{16}$.
\end{remark}

\begin{lemma}\label{lem:deg3irr}
  The classes of degree $3$ (we use the notion of degree from the definition of beautiful formulae, Section~\ref{sec:beau}) have the following decompositions in irreducible representations:
\begin{align*} 	 
     [S^{(3)}] &= 1 + (1 + \X_3 )\L + (3 + \X_3 + \X_10 )\L^2 + (3 + 3\X_3 + 2\X_10 + \X_16 )\L^3 +\\
     &+ (3 + \X_3 + \X_10 )\L^4 + (1 + \X_3 )\L^5 + \L^6, \\ 
[S\times S^{(2)}] &= 1 + (2 + 2\X_3 )\L^1 + (6 + 3\X_3 + \X_9 + 2\X_10 )\L^2 +\\
               &+ (6 + 7\X_3 + \X_9 + 3\X_10 + \X_16 + \X_20 )\L^3 + \\
	       &+ (6 + 3\X_3 + \X_9 + 2\X_10 )\L^4 + (2 + 2\X_3 )\L^5 + \L^6, \\
         [S^3] &= 1 + (3 + 3\X_3 )\L + (9 + 6\X_3 + 3\X_9 + 3\X_10 )\L^2 +\\
               &+(10 + 12\X_3 + 3\X_9 + 4\X_10 + \X_12 + \X_16 + 2\X_20 )\L^3 + \\
	       &+ (9 + 6\X_3 + 3\X_9 + 3\X_10 )\L^4 + (1 + \X_3 )\L^5 + \L^6.  
\end{align*}
\end{lemma}
\begin{proof}
  The calculations are analogous to calculations in Lemma~\ref{lem:deg2irr}.
\end{proof}

\begin{theorem}\label{No formulae of degree 3}
        There are no beautiful formulae of degree $3$ for smooth cubic surfaces $S$ with the Fano variety $F(S)$ and the LLSvS variety $Z(S)$.
\end{theorem}
\begin{proof}
  A beautiful formula of degree $3$ implies a formula of degree $3$ in $\Rep(\W, \C)$, Lemma~\ref{lem:subring}.
  Any formula of degree $3$ uses some classes from Lemma~\ref{lem:deg3irr}. 
  These classes have new irreducible summands in comparing with the classes of degree $2$ and $1$ and the class $[F(S)]$.
  The class~$[S^{(3)}]$ has $\chi_{16}$. 
  The class $[S\times S^{(2)}]$ has $\chi_{20}$.
  The class~$[S^3]$ has $\chi_{12}$.
  That is why any such formula uses the class~$[Z(S)]$. 
  This class has the new irreducible summand $\chi_8$ in comparing with the other classes, Lemma~\ref{lem:Zirr}.
\end{proof}


\begin{lemma}\label{lem:deg4irr}
  The classes of degree $4$ (we use the notion of degree from the definition of beautiful formulae, Section~\ref{sec:beau}) have the following decompositions in irreducible representations:
\begin{align*}
           [S^{(4)}] &= 1 + (1 + \X_3  )\L + (3 + \X_3 + \X_10 )\L^2 + (3 + 3\X_3 + 2\X_10 + \X_16 )\L^3 +\\&+  (6 + 4\X_3 + \X_8 + 5\X_10 + \X_16 + \X_20 )\L^4 +
	   (3 + 3\X_3 + 2\X_10 + \X_16 )\L^5 +\nonumber\\&+ (3 + \X_3 + \X_10 )\L^6 + (1 + \X_3 )\L^7 +\L^8, \nonumber\\ 
	   [S\times S^{(3)}] &=  1 + (2 + 2\X_3 )\L + (6 + 3\X_3 + \X_9 + 2\X_10 )\L^2 +\\&+ (8 + 9\X_3 + \X_9 + 5\X_10 + 2\X_16 + \X_20 )\L^3 +\nonumber\\&+ (12 + 10\X_3 + \X_8 + 3\X_9 + 10\X_10 + 3\X_16 + 3\X_20 + \X_23 )\L^4 +\nonumber\\&+ (8 + 9\X_3 + \X_9 + 5\X_10 + 2\X_16 + \X_20 )\L^5 +\nonumber\\&+ (6 + 3\X_3 + \X_9 + 2\X_10 )\L^6 + (2 + 2\X_3 )\L^7 + \L^8,\nonumber\\
          [S^4] &= 1 + (4\X_3 + 4 )\L + (12\X_3 + 6\X_9 + 6\X_10 + 16 )\L^2 +\\
                &+ (36\X_3 + 12\X_9 + 16\X_10 + 4\X_12 + 4\X_16 + 8\X_20 + 28 )\L^3 +\nonumber\\
                &+ (41\X_3 + \X_7 + \X_8 + 24\X_9 + 29\X_10 + 4\X_12 + 2\X_13 + 7\X_16 + 2\X_17 + 12\X_20 +\nonumber\\&+ 3\X_23 + 3\X_25 + 40 )\L^4 +\nonumber\\
                &+ (36\X_3 + 12\X_9 + 16\X_10 + 4\X_12 + 4\X_16 + 8\X_20 + 28 )\L^5 +\nonumber\\ 
		&+ (12\X_3 + 6\X_9 + 6\X_10 + 16 )\L^6 + (4\X_3 + 4 )\L^7 + \L^8,\nonumber\\
           [S^{(2 )}\times S^{(2 )}] &= 1 + (2\X_3 + 2 )\L + (4\X_3 + \X_9 + 3\X_10 + 8 )\L^2 +\\
                                     &+ (12\X_3 + 2\X_9 + 6\X_10 + 2\X_16 + 2\X_20 + 10 )\L^3 +\nonumber\\
                                     &+ (13\X_3 + \X_8 + 4\X_9 + 13\X_10 + \X_13 + 3\X_16 + \X_17 + 4\X_20 + \X_23 + 17 )\L^4 +\nonumber\\
                                     &+ (12\X_3 + 2\X_9 + 6\X_10 + 2\X_16 + 2\X_20 + 10 )\L^5 +\nonumber\\
				     &+ (4\X_3 + \X_9 + 3\X_10 + 8 )\L^6 + (2\X_3 + 2 )\L^7 + \L^8,\nonumber\\
				     [S^2\times S^{(2 )}] &= 1 + (3\X_3 + 3 )\L + (7\X_3 + 3\X_9 + 4\X_10 + 11 )\L^2 +\\&+ (21\X_3 + 5\X_9 + 10\X_10 + \X_12 + 3\X_16 + 4\X_20 + 17 )\L^3 +\nonumber\\&+ (23\X_3 + \X_8 + 11\X_9 + 19\X_10 + \X_12 + \X_13 + 5\X_16 + \X_17 + 7\X_20 +\nonumber\\&+ 2\X_23 + \X_25 + 25 )\L^4 +\nonumber\\&+ (21\X_3 + 5\X_9 + 10\X_10 + \X_12 + 3\X_16 + 4\X_20 + 17 )\L^5 +\nonumber\\&+ (7\X_3 + 3\X_9 + 4\X_10 + 11 )\L^6 + (3\X_3 + 3 )\L^7 + \L^8. \nonumber
\end{align*}
\end{lemma}
\begin{proof}
  The calculations are analogous to calculations in Lemma~\ref{lem:deg2irr}.
\end{proof}

\begin{proposition}\label{prop:nodeg4hom}
        There are no homogeneous beautiful formulae of degree $4$ for smooth cubic surfaces $S$.
\end{proposition}
\begin{proof}
  Any homogeneous beautiful formula of degree $4$ for smooth cubic surfaces $S$ implies a formula of degree $4$ that holds in the ring $\Rep(\W, \C)$.
  We obtain irreducible decompositions for all the classes that can be involved in such a formula in Lemma~\ref{lem:deg2irr}, Lemma~\ref{lem:deg3irr}, Lemma~\ref{lem:deg4irr}.
  Simple linear algebra calculations shows that there are no relations (with coefficients in $\C[\L]$) between these classes.
  Note that it is enough to show that there are no relations after evaluation~$\L\mapsto 1$.
\end{proof}
\begin{remark}\label{Formulae of degree 5}
        The same kind of calculations can show that there is a unique (up to multiplication) homogeneous formulae of degree $5$ for smooth cubic surfaces $S$ in the ring $\Rep(\W)$:
        \begin{align}\label{Homrel5}
[S^{(5)}] &+ [S\times S^{(3)}] = [S\times S^{(4)}] + [S^{(3)}] 
   -\nonumber\\&- \L[S^2\times S^{(2)}] -\L^2[S\times S^{(3)}] + (2\L + \L^2 + 2\L^3)[S\times S^{(2)}] + (\L + \L^2 + \L^3)[S^3] 
   -\nonumber\\&- (2\L + 3\L^2 + 5\L^3 + 3\L^4 + 2\L^5)[S^{(2)}] + (\L^2 + \L^4)[S^{(3)}] 
   -\\&- (\L+\L^2+\L^3 +\L^4+\L^5)[S^{(2)}] + (\L^7 + 3\L^6 + 4\L^5+ 5\L^4 + 4\L^3 + 3\L^2 + \L)[S] 
   -\nonumber\\&- (\L^8 + \L^7 + 2\L^6 + \L^5 + 2\L^4 + \L^3 + \L^2).\nonumber
\end{align}
\end{remark}
\mainthm
\begin{proof}
  A beautiful formula of degree $4$ for smooth cubic surfaces $S$ with the LLSvS variety~$Z(S)$ by Lemma~\ref{lem:evaluation} gives a formula in $\Rep(\W, \C)$.
  The relation \eqref{S-Z(S) formula} holds in $\Rep(\W, \C)$. 
  This can be checked using irreducible decompositions we provided in Lemma~\ref{lem:deg2irr}, Lemma~\ref{lem:Zirr}, Lemma~\ref{lem:deg3irr} and Lemma~\ref{lem:deg4irr}.
  By Proposition~\ref{prop:nodeg4hom} there are no relations (with coefficients in $\C[\L]$) between the classes with symmetric powers of $S$ of degree up to $4$ (we use the notation of degree from the definition of beautiful formulae, Section~\ref{sec:beau}).
  It follows that there exists a unique minimal relation with the class $[Z(S)]$, i.e. all other relations are multiplication of this one by some non-invertible polynomial in $\L$.
  Note that there is no irreducible summand $\chi_8$ in the irreducible decompositions of the classes with symmetric powers of $S$ before $\L^0$, $\L$, $\L^2$ and $\L^3$. 
  That means that the coefficient between $[Z(S)]$ in any formula with these classes is divided by $\L^4$. 
  So the relation \eqref{S-Z(S) formula} is minimal.

  By Lemma~\ref{lem:subring}, there is an embedding of the rings $\Rep(\W,\C)\cong \Rep(\W,\Q)\subset \K_0(\Chow)$. That is why the relation \eqref{S-Z(S) formula} is also a unique (up to multiplication) formula in $\K_0(\Chow)$.
  
The G\"ottsche formula \eqref{Gottsche-formula} allows to rewrite the $S$-$Z(S)$ relation in $\Hilb$-form.
\begin{align*}
  [S^{[2]}] &= 1 + (2 + \X_3 ) \L^1 + (4 + 2\X_3 + \X_10 ) \L^2 + (2 + \X_3 ) \L^3 + \L^4, \\
  [S^{[3]}] &= 1 + (2 + \X_3 ) \L^1 + (5 + 3\X_3 + \X_10 ) \L^2 + (7 + 5\X_3 + \X_9 + 3\X_10 + \X_16 ) \L^3 + \\ &+ (5 + 3\X_3 + \X_10 ) \L^4 + (2 + \X_3 ) \L^5 + \L^6, \\
  [S^{[4]}] &= 1 + (2 + \X_3 ) \L^1 + (6 + 3\X_3 + \X_10 ) \L^2 + (10 + 7\X_3 + \X_9 + 4\X_10 + \X_16 ) \L^3 + \\ &+ (15 + 12\X_3 + \X_8 + \X_9 + 9\X_10 + 2\X_16 + 2\X_20 ) \L^4 + (10 + 7\X_3 + \X_9 + 4\X_10 + \\ &+ \X_16 ) \L^5 + (6 + 3\X_3 + \X_10 ) \L^6 + (2 + \X_3 ) \L^7 + \L^8, \\
  [S\times S^{[2]}] &= 1 + (3 + 2\X_3 ) \L^1 + (8 + 5\X_3 + \X_9 + 2\X_10 ) \L^2 + \\ &+ (10 + 9\X_3 + 2\X_9 + 4\X_10 + \X_16 + \X_20 ) \L^3 + \\ &+ (8 + 5\X_3 + \X_9 + 2\X_10 ) \L^4 + (3 + 2\X_3 ) \L^5 + \L^6.
\end{align*}
\end{proof}
\motivicmainthm
\begin{proof}
  The category of rational graded representations of $\W$ is semi-simple. That is why all the decompositions in irreducible representations above are equivalences in this category. Since it is a subcategory in $\Chow$ the relation in the theorem is a corollary of these decompositions. 
\end{proof}

\subsection{Singular cubic surfaces} \label{Singular cubic surfaces}
Let $Y$ be a cubic hypersurface of dimension $d$ over a field $k$. Note that the $Y$-$F(Y)$ relation (\ref{The Y-F(Y) relation}) in the $\Sym$-form in the case of singular cubic hypersurfaces has the following form:
$$[Y^{(2)}] = (1 + \L^d)[Y] + \L^2[F(Y)] - \L^d[Sing(Y)]$$
where $Sing(Y)$ is the singular locus of $Y$, \cite[Theorem~5.1]{GS14}. 
The $S$-$Z(S)$ relation also may have similar forms in the case of singular surfaces.
In this section we will investigate the case of singular cubic surfaces with at most rational double point singularities. We obtain some analogs of the $S$-$Z(S)$ relation in $\KGro$ in the cases of singularities of type $A_1$ and $A_2$.

Let $S$ be a singular cubic surface over a field $k$ with one rational double point singularity and let $\tilde{S}\rightarrow S$ be its minimal resolution.
The smooth surface $\tilde{S}$ is a weak Del Pezzo surface.
The structure on $H^2(\tilde{S}, \Z)$ is the same as in the smooth case, Section \ref{The 27 lines}.
The orthogonal complement $K_{\tilde{S}}^{\perp}\subset H^2(\tilde{S},\Z)$ is the negative definite root lattice of type $E_6$.
Curves in the exceptional divisor generate a root sublattice $R_0\subset R$.
See~\cite[Section~2.1]{LLSvS} for detail account about possible singularities of cubic surfaces and corresponding structures on cohomology and the book \cite[Section~9.2]{Dol12} for detailed descriptions of singular cubic surfaces.



\subsubsection{Type $A_1$} \label{A1}
Let $S$ be a cubic surface over a field $k$ with one rational double point singularity of type $A_1$. 
The root system $R_0$ consists of two roots $\alpha, -\alpha$.
Fix the notation from Section~\ref{The 27 lines}.
Without loss of generality we may assume that $\alpha = 2E_0 - \sum E_i$, $i\in{1,\dots,6}$.
The orthogonal complement to $R_0$ is formed by the roots $E_i-E_j$, $i\neq j$. 
So it is a root system of type $A_5$. 
The subgroup $G\subset \W$ that permutes $E_i$ is $\S_6$.

According to \cite{Dol12} the surface $S$ can be obtained in the following way.
We can find a smooth conic $C\subset \P^2$ and a smooth subscheme $A(S)$ of length $6$ on $C$, such that $\tilde{S}$ is the blow up of $A(S)$ and $S$ is the blow down of $C$ considered as curve on $\tilde{S}$.
The scheme $A(S)$ is the scheme of lines passing through the singular point.
Sometimes we will drop the dependence on~$S$ in the notation $A(S)$ for more compact equations. 

\Aonegrolemma
\begin{proof}
  Using the construction with the blow up we see that $[S] = \P^2 + [A(S)]\L - [C]+1$. 
\end{proof}

\begin{remark}
        According to Section 9.2 \cite{Dol12} a singular cubic surface $S$ with one rational double point singularity of type $A_1$ has the following explicit description.
        Let $p = [1, 0, 0 ,0]$ be the singular point of $S = V(f_3)$.
        Then $f_3$ can be written as:
        $$f_3 = t_0 g_2(t_1,t_2,t_3) + g_3(t_1,t_2,t_3),$$
        where $g_2$, $g_3$ are homogeneous polynomials of degree 2 and 3.
        And $V(g_2)\subset \P^2$ is a nonsingular conic which intersects $V(g_3)$ transversally. 
        The variety $A(S)$ is this intersection. 
\end{remark}

In this section we will concentrate on the case when $\Delta = 0$. It is equivalent to the condition that $C$ has a $k$-point. 
The scheme $A$ is completely described by the action of the absolute Galois group $\Gal_k$ on its geometric points.
This action factors through the action of $G$.
This gives a map $\Burn(G)\rightarrow \Burn(\Gal_k)$, see Section~\ref{sec:burnside} for details on Burnside rings.
Denote by~$[A]$ the class of the $G$-set of geometric points of $A$ in $\Burn(G)$.

According to Theorem 2.1 \cite{LLSvS} the moduli space of generalized twisted cubics on~$S$ has~$R/\W(R_0)$ copies of $\P^2$ and the $\Gal_k$-set of geometric points of $Z(S)$ is isomorphic to~$R/\W(R_0)$, Section~\ref{LLSvS}.
Note that $\W(R_0)$ is isomorphic to $\Z_2$.
The action of $\Gal_k$ on $R/\W(R_0)$ factors through the action of $G$.
Denote by $[Z(S)]$ the corresponding $G$-set of geometric points of $Z(S)$.

\begin{lemma}\label{lem:ZA1}
  The following relation between $A$ and $Z(S)$ holds in $\Burn(G)$ and in $\KGro$:
\begin{equation} \label{Z-A1}
  [Z(S)] = \One + [A^{(3)}] - [A].
\end{equation}
\end{lemma}
\begin{proof}
  The relation in $\KGro$ will follow from the relation in $\Burn(G)$ due to the sequence of maps of pre-$\lambda$-rings, Section~\ref{sec:burnside}:
  \[ \Burn(G) \rightarrow \Burn(\Gal_k) \rightarrow \KGro. \]

To prove the relation in $\Burn(G)$ we need to analyze the $G$-set $R/\W(R_0)$.
Divide 72 roots of~$R$ in five groups as in Section \ref{The 27 lines}:
\begin{enumerate}
        \item One vector $2E_0 - \sum E_i$.
        \item Twenty vectors $E_0 - E_i - E_j - E_k$, $i\neq j\neq k\neq 0$.
        \item Thirty vectors $E_i - E_j$, $i\neq j$.
        \item Twenty vectors $-E_0 + E_i + E_j + E_k$, $i\neq j\neq k\neq 0$.
        \item One vector $-2E_0 + \sum E_i$.
\end{enumerate}
The only nontrivial element in $\W(R_0)=\Z_2$ is the reflection in the root $\alpha$ defined by the formula:
\[ r_{\alpha}(\beta) = \beta + (\alpha, \beta) \alpha. \]
The groups (1) and (5) give one $\W(R_0)$-orbit and $G$ acts trivially on it.
The group (3) consists of trivial $\W(R_0)$-orbits. 
The $G$-set of these orbits is the $G$-set of pairs of different points on~$A$.
Its class in $\Burn(G)$ is $[A^2]-[A]$.
The action of $\W(R_0)$ swaps the groups (2) and (4). 
The $G$-set of corresponding $\W(R_0)$-orbits is the set of unordered triples of different points on $A$.
Its class in $\Burn(G)$ is $[A^{(3)}] - [A^2]$.
\end{proof}

\begin{remark}
        The same analysis of the explicit description of lines on $S$ from Section \ref{The 27 lines} can show that:
	\[ \KGro\ni [F(S)] = [A^{(2)}].\]
\end{remark}

\begin{lemma}\label{lem:deg3eq}
The following relation holds in the Grothendieck ring of varieties $\KGro$:
\begin{equation} \label{eq:SZS3A1}
  [S^{(3)}] - [S^{(2)}] - [S^{(2)}]\L^2 + \L^2 + \L^3 + \L^4 = [Z(S)] \L^3.
\end{equation}
\end{lemma}
\begin{proof}
Consider the classes of symmetric powers $S^{(2)}$, $S^{(3)}$ in $\KGro$:
\begin{align}
  [S^{(2)}] &= \One + [A] \L + (\One + [A^{(2)}])\L^2 + [A] \L^3 + \L^4, \label{eq:SS2A1} \\
	  [S^{(3)}] &= \One + [A] \L + (\One + [A^{(2)}])\L^2 + ([A] + [A^{(3)}])\L^3 + (\One + [A^{(2)}])\L^4 + [A]\L^5 + \L^6. \nonumber
\end{align}

These equations and Lemma~\ref{lem:ZA1} imply \eqref{eq:SZS3A1}.
\end{proof}

We want to find a relation similar to the S-Z(S) relation \eqref{S-Z(S) formula}.
It can be shown using calculations with characters that there exists exactly one (up to multiplication) homogeneous formula of degree $4$ with $S$ in the ring $\Rep(G,\C)$.

\begin{lemma}\label{lem:A1deg4homeq}
  The following relation holds in the Grothendieck ring of varieties $\KGro$:
\begin{align} \label{eq:SS4A1}
  [S^{(4)}] &- [S^{(3)}](\One + \L^2) - [S^{(2)}][S]\L + [S^2](\L+ \L^2+\L^3)+\\&+ [S^{(2)}](\L -\L^2 + \L^3) - [S](\L + 2\L^3 + \L^5) +\L^2 + \L^6 = 0.\nonumber
\end{align}
\end{lemma}
\begin{proof}
  Substituting the expressions (such as \eqref{eq:SS2A1}) in $\KGro$ of the classes of symmetric powers of $S$ through $[A]$ in \eqref{eq:SS4A1} we obtain:
\begin{equation} \label{A equation for type A_1}
  [A^{(4)}] + [A^2] + [A] = [A][A^{(2)}] + 2 [A^{(2)}].
\end{equation}

This equation is easy to check in $\Burn(G)$. 
Denote by $[A^{\{n\}}]$ the variety of unordered $n$-tuples of different points on $A$. 
Then in $\Burn(G)$ we have:
\begin{equation} \label{Symmetric powers of A for type A_1}
        \begin{split}
	  [A^{(2)}] &= [A] + [A^{\{2\}}], \\ 
	  [A^{(4)}] &= [A^{\{4\}}] + [A][A^{\{2\}}] - [A][A] + [A] + [A^{\{2\}}] + [A][A] - [A] + [A] \\&= [A][A^{\{2\}}] + [A] + 2[A^{\{2\}}].
        \end{split}
\end{equation}

Note that there exists a duality $[A^{\{6-n\}}] = [A^{\{n\}}]$ in $\Burn(G)$. Equations (\ref{Symmetric powers of A for type A_1}) gives Equation~\eqref{A equation for type A_1}.
\end{proof}

\Aonethm
\begin{proof}
This relation is the combination of Equation~\eqref{eq:SS4A1} with Equation~\eqref{eq:SZS3A1}.
\end{proof}
\begin{remark}
        This relation is similar to the $S$-$Z(S)$ relation for smooth $S$:
\begin{equation*}
  \begin{split}
   \L^4 [Z] &= [S^{(4)}] - (1 -  \L +  \L^2)[S^{(3)}] -  
   \L \HY [S^{(2)}] + ( \L+ \L^2+ \L^3)[S^2] - \\ &-
   2\L^2[S^{(2)}] - ( \L- \L^2+ \L^3- \L^4+ \L^5) \HY + ( \L^2+ \L^4+ \L^6).
  \end{split}
\end{equation*}
    The difference is $([A^{\{2\}}]- 1)\L^4$. 
\end{remark}

\subsubsection{Type $A_2$.} \label{A2} Let $S$ be a cubic surface over a field $k$ with one rational double singularity of type $A_2$.
Consider a singular conic $C\subset P^2$ that is an intersection of two lines over algebraic closure.
Denote corresponding zero-dimensional scheme of this two lines as $K$.
Consider a zero-dimensional reduced subscheme $A(S)\subset C$ that become a two triples of points over algebraic closure. 
One triple on the each line of $C$.
Further we will omit the dependence on $S$ in the notation $A(S)$.

According to \cite{Dol12} one can choose $A(S)$ and $C$ such that the surface $S$ can be obtained in two steps. 
First, blow up $\P^2$ in $A(S)$.
Then blow down the conic.
Denote by $E_i$ the classes of the exceptional curves on $S_{\bar{k}}$ such that
$E_1, E_2, E_3$ correspond to three points on one line from~$C$ and $E_4, E_5, E_6$ correspond to the other line.

\begin{lemma}\label{lem:SA2}
Using the notations introduced above we have:
\begin{equation}\label{eq:SAK}
  \KGro \ni [S] = 1 + (1 + [A] - [K])\L + \L^2.
\end{equation}
\end{lemma}
\begin{proof}
	The class of the singular conic $[C] = [K]\L + 1$.
	From the construction with the blow up and blow down we have $[S] = [P^2] + [A]\L - [C] + 1$. 
\end{proof}

The construction of $S$ gives the following description of the root lattice $R_0\subset R$.
It consists of the roots:
\begin{align*}
  2E_0-\sum E_i,\ \ E_0 - E_1 - E_2 - E_3,\ \ E_0 - E_4 - E_5 - E_6, \\ 
  \sum E_i - 2E_0,\ \ E_1 + E_2 + E_3 - E_0,\ \ E_4 + E_5 + E_6 - E_0.
\end{align*}
The orthogonal root lattice to $R_0$ is of type $A_2\times A_2$.
The first $A_2$ is generated by the roots~$E_1,E_2,E_3$.
The second $A_2$ is generated by the roots~$E_4,E_5,E_6$. 
Consider the subgroup~$\Z_2\ltimes(\S_3)^2=:G\subset \S_6$, where $\S_6\subset \W$ permutes $E_i$, $i\neq 0$.  
The first summand $\S_3$ permutes $E_1,E_2,E_3$. 
The second summand $\S_3$ permutes $E_4,E_5,E_6$. 
The semi-summand $\Z_2$ swaps these triples of classes.
The action of the absolute Galois group $\Gal_k$ on $H^2(S,\Z)$ factors through $G$.
The action of $G$ on $H^2(S,\Z)$ induces the action on geometric points of $A$ and~$K$.
The action of the absolute Galois group $\Gal_k$ on the geometric points of $A$ and $K$ factors through the action of $G$, since $G$ maps surjectively on the automorphism groups of the sets of geometric points of $A$ and $K$. 

In other words we have the map of the Burnside rings $\Burn(G)\rightarrow \Burn(\Gal_k)$ induced by the map $\Gal_k\rightarrow G$.
Denote by $[A]$ and $[K]$ the classes of $A$ and $K$ in $\Burn(G)$.
The action of the Galois group $\Gal_k$ on the geometric points of $Z(S)$ by Theorem~\ref{gtc on sing surf} identifies with the action on $R/\W(R_0)$ where $\W(R_0)$ is the Weyl group of the root system $R_0$.
It factors through the action of $G$ on $R/\W(R_0)$.
Denote by $[Z(S)]\in \Burn(G)$ the class of the corresponding $G$-set.



Let us describe explicitly the action of $G$ on $[K]\in \Burn(G)$.
The subgroups $S_3$ act trivially and the subgroup $\Z_2$ acts by permutation of these two points.
The $G$-set $[A]\in \Burn(G)$ has the following description.
It consists of two triples of points corresponding to different lines in~$C$.
The subgroups $S_3$ permutes these triples of points. 
The first one permutes the first triple and the second one permutes the second triple.
The subgroup $\Z_2$ swaps these triples.

Denote by $\Rep(G,\C)$ the ring of graded complex representations of $G$.
Denote by $[A]$, $[K]$ the corresponding permutational representations of $G$ constructed from the corresponding $G$-sets.
Denote by~$\L$ the trivial representation with the grading $1$.
Since $[S]$ in \eqref{eq:SAK} is expressed through~$[A], [K]$ and $\L$ we can consider $[S]$ as a class in $\Rep(G,\C)$.
Now we are going to obtain an analog of the $S$-$Z(S)$ relation \eqref{S-Z(S) formula} in $\Rep(G)$ and then we will prove it in $\KGro$. 
The character table for the group $G$ is given in Appendix Table~\ref{Z2S3S3}.
The conjugation classes are denoted by their representatives in $\S_6$. 
Denote by $\ch W$ the character of a $G$-representation~$W$ in notation of Table~\ref{Z2S3S3}.

\begin{lemma} 
Denote by $V\in\Rep(G,\C)$ the representation $1+[A]-[K]$, then 
$$ V = 1 + \chi_9 $$
\end{lemma}
\begin{proof}
Since the representations $[A]$ and $[K]$ are permutational it is easy to calculate their characters:
\begin{align*}
  \ch [A] &= \left(6, 4, 3, 2, 1, 0, 0, 0, 0\right), \\
  \ch [K] &= \left(2, 2, 2, 2, 2, 2, 0, 0, 0\right). 
\end{align*}
We have:
$$ \ch [V] = \left(5, 3, 2, 1, 0, -1, 1, 1, 1\right). $$
Using the table of characters (Table~\ref{Z2S3S3}) this representation identifies as $1+\chi_9$.
\end{proof}

To identify $[Z(S)]\in\Rep(G,\C)$ we need to understand the action of $G$ on $R/\W(R_0)$, Section~\ref{LLSvS}.
Denote by $A_{12}\subset A\times A$ the scheme of pairs of points lying on the same line on~$C$ and by~$[A_{12}]$ the classes of the corresponding $G$-set in $\Burn(G)$ and $\Rep(G,\C)$.  
Denote \hbox{by $A_{18}\subset A\times A$} the scheme of pairs of points lying on the different lines in $C$ and by $[A_{18}]$ the classes of the corresponding $G$-set in $\Burn(G)$ and $\Rep(G,\C)$.  

\begin{lemma} We have:
  \begin{align*}
    \Burn(G) \ni [Z(S)] &= 1 + [A_{12}] + [A_{18}], \\
    \ch [Z(S)] &= (31, 12, 6 , 2, 0, 0 , 0, 0, 0).
  \end{align*}
\end{lemma}
\begin{proof}
Divide $72$ roots of $R$ in five sets as we do in type $A_1$.
We need to understand an action of $G$ on the $\W(R_0)$-orbits in $R$.
By \cite[Section~3.1]{LLSvS} we know that there are $31$ orbits. 
There is a big orbit $R_0$ and $G$ does not move this orbit.
From the roots of the form $E_i-E_j$ twelve roots $E_i-E_j$, $\{i,j\}\subset\{1,2,3\}$ or $\{i,j\}\subset\{3,4,5\}$, lie in different $\W(R_0)$-orbits and form one $G$-orbit.
This gives a subscheme in $[Z(S)]$ isomorphic to $A_{12}$.
We have:
$$ \ch [A_{12}] = \left(12, 6, 6, 0, 0, 0, 0, 0, 0\right) $$

Other $18$ roots of the same form also lie in different $\W(R_0)$-orbits and form another~$G$-orbit.
This gives a subscheme in $[Z(S)]$ isomorphic to $A_{18}$.
We have:
$$ \ch [A_{18}] = \left(18, 6, 0, 2, 0, 0, 0, 0, 0\right)$$
It is the full list of orbits.
\end{proof}
After calculations with characters of symmetric powers one can obtain that a subspace generated by $1$, $V$, $V^2$, $V^3$, $V^4$, $\Sym^2V$, $\Sym^3V$,$\Sym^4V$, $\Sym^2V\times\Sym^2V$, $\Sym^2V\times V$, $\Sym^2V\times V^2$, $\Sym^3V\times V$ does not contain $[Z]$.
The situation changes when we add the class $[A]$.
 \Atwothm

\begin{proof}
Consider the set of geometric points of $A$. 
It has two $\S_3\times\S_3$-orbits.
Denote by~$A_{9}\subset A^{(2)}$ the variety of unordered pairs of points on $A$ lying in the different $\S_3\times \S_3$-orbits.
The variety~$A^{(2)}$ as a $G$-set has three $G$-orbits:
\begin{enumerate}
\item The orbit formed by pairs of equal points.
This~$G$-orbit is isomorphic to the $G$-set~$A$.
\item The orbit formed by pairs of different points in the same $\S_3\times \S_3$-orbit.
This $G$-orbit is also isomorphic to the $G$-set $A$.
\item The orbit formed by pairs of different points in different $\S_3 \times \S_3$-orbits.
This $G$-orbit is isomorphic to the $G$-set $A_{9}$.
\end{enumerate}
We have:
\begin{equation}
  \Burn(G) \ni [A^{(2)}] = [A] + [A] + [A_{9}].
\end{equation}
The variety $A^{(3)}$ as a $G$-set has five orbits:
\begin{enumerate}
  \item The orbit formed by triples of equal points.
This $G$-orbit is isomorphic to the $G$-set $A$.
\item The orbit formed by triples of two equal points and one other point in the same $\S_3\times\S_3$-orbit.
This $G$-orbit is isomorphic to the $G$-set $A_{12}$.
\item The orbit formed by triples of two equal points and one other point lying in different $\S_3\times\S_3$-orbits.
This $G$-orbit is isomorphic to the $G$-set $A_{18}$.
\item The orbit formed by triples of three different points in the same $\S_3\times \S_3$-orbit.
This $G$-orbit is the $G$-set $K$.
\item The orbit formed by triples of three different points with one of them in the other $\S_3\times\S_3$-orbit.
This $G$-orbit is isomorphic to the $G$-set $A_{18}$.
\end{enumerate}
We have:
\begin{equation}
  \Burn(G) \ni [A^{(3)}] = [K] + 2[A_{18}] + [A_{12}] + [A].
\end{equation}
The similar analysis shows that:
\begin{equation}
  \Burn(G) \ni [A^{(4)}] = 4[A] + 2[A_9] + [A_{12}] + [A_{18}] + [A][A_9].
\end{equation}
The symmetric powers of $K$ is much simpler:
\begin{align*}
  \Burn(G) \ni [K^{(2)}] &= [K] + 1, \\
  \Burn(G) \ni [K^{(3)}] &= [K] + [K], \\ 
  \Burn(G) \ni [K^{(4)}] &= [K] + [K] + 1. 
\end{align*}
We also need:
\begin{align*}
  \Burn(G) \ni [A^2] &= [A] + [A_{12}] + [A_{18}], \\
  \Burn(G) \ni [K^n] &= 2^{n-1}[K].
\end{align*}
Substitution of all the above expressions in Equation \eqref{eq:SZSA2} gives $(2[A_{18}] - [K][A_{18}])\L^4$. But the $G$-sets $A_{18} \sqcup A_{18}$ and $K\times A_{18}$ are isomorphic. So Equation \eqref{eq:SZSA2} holds in $\KGro$.
\end{proof}

\newpage
\addtocontents{toc}{\protect\setcounter{tocdepth}{1}}
\section*{Appendix}
\begin{center}
\hspace{-1.0cm}
\begin{minipage}{1.0\textwidth}        
        \captionof{table}{Character table for the Weyl group of type $E_6$}\label{E6}
\begin{center}\tiny\fontsize{8}{10}\selectfont
\renewcommand{\arraystretch}{1.1}
\begin{tabular}{ c | c c c c c c c c c c c c c c c c c c c c c c c c c }
  Class  &  1 & 2 & 3 & 4 & 5 & 6 & 7 & 8 & 9 & 10 & 11 & 12 & 13 & 14 & 15 & 16 & 17 & 18 & 19 & 20 & 21 & 22 & 23 & 24 & 25 \\
    Order  &  1 & 2 & 2 & 2 & 2 & 3 & 3 & 3 & 4 & 4 & 4 & 4 & 5 & 6 & 6 & 6 & 6 & 6 & 6 & 6 & 8 & 9 & 10 & 12 & 12  \\
  \hline
    $ p = 2 $ & 1 & 1 & 1 & 1 & 1 & 6 & 7 & 8 & 3 & 4 & 4 & 4 & 13 & 6 & 7 & 7 & 8 & 8 & 7 & 8 & 9 & 22 & 13 & 19 & 14  \\
    $ p = 3 $ & 1 & 2 & 3 & 4 & 5 & 1 & 1 & 1 & 9 & 10 & 11 & 12 & 13 & 3 & 3 & 2 & 3 & 2 & 4 & 5 & 21 & 6 & 23 & 10 & 9  \\
    $ p = 5 $ & 1 & 2 & 3 & 4 & 5 & 6 & 7 & 8 & 9 & 10 & 11 & 12 & 1 & 14 & 15 & 16 & 17 & 18 & 19 & 20 & 21 & 22 & 2 & 24 & 25  \\
  \hline
    $\chi_{1}$  &  1 & 1 & 1 & 1 & 1 & 1 & 1 & 1 & 1 & 1 & 1 & 1 & 1 & 1 & 1 & 1 & 1 & 1 & 1 & 1 & 1 & 1 & 1 & 1 & 1  \\
    $\chi_{2}$  &  1 & -1 & 1 & 1 & -1 & 1 & 1 & 1 & 1 & -1 & -1 & 1 & 1 & 1 & 1 & -1 & 1 & -1 & 1 & -1 & -1 & 1 & -1 & -1 & 1  \\
    $\chi_{3}$  &  6 & 4 & -2 & 2 & 0 & -3 & 3 & 0 & 2 & -2 & 2 & 0 & 1 & 1 & 1 & 1 & -2 & -2 & -1 & 0 & 0 & 0 & -1 & 1 & -1  \\
    $\chi_{4}$  &  6 & -4 & -2 & 2 & 0 & -3 & 3 & 0 & 2 & 2 & -2 & 0 & 1 & 1 & 1 & -1 & -2 & 2 & -1 & 0 & 0 & 0 & 1 & -1 & -1  \\
    $\chi_{5}$  &  10 & 0 & -6 & 2 & 0 & 1 & -2 & 4 & 2 & 0 & 0 & -2 & 0 & -3 & 0 & 0 & 0 & 0 & 2 & 0 & 0 & 1 & 0 & 0 & -1  \\
    $\chi_{6}$  &  15 & -5 & 7 & 3 & -1 & -3 & 0 & 3 & -1 & -3 & 1 & 1 & 0 & 1 & -2 & -2 & 1 & 1 & 0 & -1 & 1 & 0 & 0 & 0 & -1  \\
    $\chi_{7}$  &  15 & -5 & -1 & -1 & 3 & 6 & 3 & 0 & 3 & -1 & -1 & -1 & 0 & 2 & -1 & 1 & 2 & -2 & -1 & 0 & 1 & 0 & 0 & -1 & 0  \\
    $\chi_{8}$  &  15 & 5 & 7 & 3 & 1 & -3 & 0 & 3 & -1 & 3 & -1 & 1 & 0 & 1 & -2 & 2 & 1 & -1 & 0 & 1 & -1 & 0 & 0 & 0 & -1  \\
    $\chi_{9}$  &  15 & 5 & -1 & -1 & -3 & 6 & 3 & 0 & 3 & 1 & 1 & -1 & 0 & 2 & -1 & -1 & 2 & 2 & -1 & 0 & -1 & 0 & 0 & 1 & 0  \\
    $\chi_{10}$  &  20 & 10 & 4 & 4 & 2 & 2 & 5 & -1 & 0 & 2 & 2 & 0 & 0 & -2 & 1 & 1 & 1 & 1 & 1 & -1 & 0 & -1 & 0 & -1 & 0  \\
    $\chi_{11}$  &  20 & -10 & 4 & 4 & -2 & 2 & 5 & -1 & 0 & -2 & -2 & 0 & 0 & -2 & 1 & -1 & 1 & -1 & 1 & 1 & 0 & -1 & 0 & 1 & 0  \\
    $\chi_{12}$  &  20 & 0 & 4 & -4 & 0 & -7 & 2 & 2 & 4 & 0 & 0 & 0 & 0 & 1 & -2 & 0 & -2 & 0 & 2 & 0 & 0 & -1 & 0 & 0 & 1  \\
    $\chi_{13}$  &  24 & 4 & 8 & 0 & 4 & 6 & 0 & 3 & 0 & 0 & 0 & 0 & -1 & 2 & 2 & -2 & -1 & 1 & 0 & 1 & 0 & 0 & -1 & 0 & 0  \\
    $\chi_{14}$  &  24 & -4 & 8 & 0 & -4 & 6 & 0 & 3 & 0 & 0 & 0 & 0 & -1 & 2 & 2 & 2 & -1 & -1 & 0 & -1 & 0 & 0 & 1 & 0 & 0  \\
    $\chi_{15}$  &  30 & -10 & -10 & 2 & 2 & 3 & 3 & 3 & -2 & 4 & 0 & 0 & 0 & -1 & -1 & -1 & -1 & -1 & -1 & -1 & 0 & 0 & 0 & 1 & 1  \\
    $\chi_{16}$  &  30 & 10 & -10 & 2 & -2 & 3 & 3 & 3 & -2 & -4 & 0 & 0 & 0 & -1 & -1 & 1 & -1 & 1 & -1 & 1 & 0 & 0 & 0 & -1 & 1  \\
    $\chi_{17}$  &  60 & 10 & -4 & 4 & 2 & 6 & -3 & -3 & 0 & -2 & -2 & 0 & 0 & 2 & -1 & 1 & -1 & 1 & 1 & -1 & 0 & 0 & 0 & 1 & 0  \\
    $\chi_{18}$  &  60 & -10 & -4 & 4 & -2 & 6 & -3 & -3 & 0 & 2 & 2 & 0 & 0 & 2 & -1 & -1 & -1 & -1 & 1 & 1 & 0 & 0 & 0 & -1 & 0  \\
    $\chi_{19}$  &  60 & 0 & 12 & 4 & 0 & -3 & -6 & 0 & 4 & 0 & 0 & 0 & 0 & -3 & 0 & 0 & 0 & 0 & -2 & 0 & 0 & 0 & 0 & 0 & 1  \\
    $\chi_{20}$  &  64 & 16 & 0 & 0 & 0 & -8 & 4 & -2 & 0 & 0 & 0 & 0 & -1 & 0 & 0 & -2 & 0 & -2 & 0 & 0 & 0 & 1 & 1 & 0 & 0  \\
    $\chi_{21}$  &  64 & -16 & 0 & 0 & 0 & -8 & 4 & -2 & 0 & 0 & 0 & 0 & -1 & 0 & 0 & 2 & 0 & 2 & 0 & 0 & 0 & 1 & -1 & 0 & 0  \\
    $\chi_{22}$  &  80 & 0 & -16 & 0 & 0 & -10 & -4 & 2 & 0 & 0 & 0 & 0 & 0 & 2 & 2 & 0 & 2 & 0 & 0 & 0 & 0 & -1 & 0 & 0 & 0  \\
    $\chi_{23}$  &  81 & 9 & 9 & -3 & -3 & 0 & 0 & 0 & -3 & 3 & -1 & -1 & 1 & 0 & 0 & 0 & 0 & 0 & 0 & 0 & 1 & 0 & -1 & 0 & 0  \\
    $\chi_{24}$  &  81 & -9 & 9 & -3 & 3 & 0 & 0 & 0 & -3 & -3 & 1 & -1 & 1 & 0 & 0 & 0 & 0 & 0 & 0 & 0 & -1 & 0 & 1 & 0 & 0  \\
    $\chi_{25}$  &  90 & 0 & -6 & -6 & 0 & 9 & 0 & 0 & 2 & 0 & 0 & 2 & 0 & -3 & 0 & 0 & 0 & 0 & 0 & 0 & 0 & 0 & 0 & 0 & -1
\end{tabular}
\end{center}
\end{minipage}
\end{center}
\bigskip

\begin{center}
\hspace{-0.5cm}
\begin{minipage}{\textwidth}
        \captionof{table}{Character table for $G=\Z_2\ltimes (S_3\times S_3)\subset S_6$}\label{Z2S3S3}
\begin{center} \tiny\fontsize{8}{10}\selectfont
\renewcommand{\arraystretch}{1.1}
        \begin{tabular}{r|*{9}{c}}
        Class & (~~) & (56) & (456) & (23)(56) & (23)(456) & (123)(456) & (14)(25)(36) & (14)(2536) & (142536) \\
        \hline
        $\chi_1$ &        1 & 1 & 1 & 1 & 1 & 1 & 1 & 1 & 1 \\
        $\chi_2$ &        1 & -1 & 1 & 1 & -1 & 1 & -1 & 1 & -1 \\
        $\chi_3$ &        1 & -1 & 1 & 1 & -1 & 1 & 1 & -1 & 1 \\
        $\chi_4$ &         1 & 1 & 1 & 1 & 1 & 1 & -1 & -1 & -1 \\
        $\chi_5$ &         2 & 0 & 2 & -2 & 0 & 2 & 0 & 0 & 0 \\
        $\chi_6$ &         4 & -2 & 1 & 0 & 1 & -2 & 0 & 0 & 0 \\
        $\chi_7$ &         4 & 0 & -2 & 0 & 0 & 1 & -2 & 0 & 1 \\
        $\chi_8$ &         4 & 0 & -2 & 0 & 0 & 1 & 2 & 0 & -1 \\
        $\chi_9$ &         4 & 2 & 1 & 0 & -1 & -2 & 0 & 0 & 0
\end{tabular}
\end{center}
\end{minipage}
\end{center}
\newpage
\subsection*{Summary of $S$-$Z(S)$ relations}
Let us summarize all the forms of the $S$-$Z(S)$ relation on one page. 

\subsubsection{The relations in $\Sym$-form}
For smooth cubic surfaces in $\K_0(\Chow)$ we have, Theorem~\ref{thm:main}:
\begin{align*}
         \L^4 [Z(S)] &= [S^{(4)}] - (1 - \L +  \L^2)[S^{(3)}] - \L[S][S^{(2)}] + (\L+ \L^2+ \L^3)[S^2] \\
	 &- 2\L^2[S^{(2)}] - ( \L- \L^2 + \L^3- \L^4 + \L^5)[S] + (\L^2+ \L^4+ \L^6).
\end{align*}

In the case of singular cubic surfaces with exactly one singular point and this point is a rational double point singularity of type $A_1$ we have in $\KGro$, Theorem \ref{thm:A1}: 
\begin{align*}
          \L^4 [Z(S)] &= [S^{(4)}] - (1 -  \L +  \L^2)[S^{(3)}] - 
   \L \HY [S^{(2)}] + ( \L+ \L^2+ \L^3)[S^2]  \\ &-
   \L^2[S^{(2)}] -  ( \L + 2 \L^3 + \L^5) \HY + ( \L^2 + \L^3 + \L^4 + \L^5 + \L^6).
\end{align*}

In the case of singular cubic surfaces with exactly one singular point and this point is a rational double point singularity of type $A_2$ we have in $\KGro$, Theorem \ref{thm:A2}: 
\begin{align*}
        \L^4[Z(S)] &= [S^{(4)}] - (1-\L+\L^2)[S^{(3)}] - \L[S][S^{(2)}] + (\L + \L^2 + \L^3)[S^2] \\
	&- \L^2 [S^{(2)}] - (\L + 2\L^3 + \L^5)[S] + \L^2 + \L^3 +\L^4[A(S)]+ \L^5 + \L^6. \nonumber
\end{align*}

For smooth cubic fourfolds \emph{hypothetical} formula in $\KGro$:
\begin{align*}
  \YZYequation.
\end{align*}

\subsubsection{The relations in $\Hilb$-form} 
For smooth cubic surfaces in $\K_0(\Chow)$ we have, Theorem~\ref{thm:main}:
\begin{align*}
        \L^4[Z(S)] &= [S^{[4]}] - (1 - \L + \L^2)[S^{[3]}] - 2\L[S][S^{[2]}] + (2\L + \L^2 + 2\L^3)[S^2]\\
	&- 3\L^2[S^{[2]}] - (\L - 2\L^2 - 2\L^4 + \L^5 ) [S] + (\L^2 + \L^4 + \L^6 ).
\end{align*}

\providecommand{\arxiv}[1]{\href{http://arxiv.org/abs/#1}{\tt arXiv:#1}}
\bibliographystyle{halpha}

\medskip
\address{
{\bf Pavel Popov}\\
National Research University Higher School of Economics, Russian Federation\\
Faculty of Mathematics and Laboratory for Mirror Symmetry and Automorphic Forms\\
6 Usacheva ul.\\
119048, Moscow \\
e-mail: {\tt lightmor@gmail.com}
}

\end{document}